\documentclass[11pt]{article}
\usepackage[letterpaper,margin=1in]{geometry}
\usepackage{setspace}
\usepackage[utf8]{inputenc}

\usepackage{subcaption}
\usepackage{physics}
\usepackage{amsmath,graphicx}
\usepackage{mathtools}
\mathtoolsset{showonlyrefs=true}

\usepackage{amsfonts}
\usepackage{amsthm}
\usepackage{enumitem}  
\usepackage{xfrac}
\usepackage[ruled,vlined]{algorithm2e}
\usepackage{authblk}

\usepackage[backref=page]{hyperref}
\newif\ifbackrefshowonlyfirst
\backrefshowonlyfirstfalse
\makeatletter
\let\BR@direct@old@hyper@natlinkstart\hyper@natlinkstart
\renewcommand*{\hyper@natlinkstart}{\phantomsection\BR@direct@old@hyper@natlinkstart}
\let\BR@direct@oldBR@citex\BR@citex
\renewcommand*{\BR@citex}{\phantomsection\BR@direct@oldBR@citex}%

\long\def\hyper@page@BR@direct@ref#1#2#3{p. \hyperlink{#3}{#1}}

\ifx\backrefxxx\hyper@page@backref
    \let\backrefxxx\hyper@page@BR@direct@ref
    \ifbackrefshowonlyfirst
    \fi
\else
    \ifbackrefshowonlyfirst
    \fi
\fi

\RequirePackage{etoolbox}
\patchcmd{\Hy@backout}{Doc-Start}{\@currentHref}{}{\errmessage{I can't seem to patch backref}}
\makeatother

\usepackage[numbers]{natbib}

\newtheorem{theorem}{Theorem}

\newtheorem{assumption}[theorem]{Assumption}
\newtheorem{lemma}[theorem]{Lemma}
\newtheorem{corollary}[theorem]{Corollary}

\newtheorem{remark}[theorem]{Remark}
\newenvironment{proofof}[1]
  {\begin{proof}[Proof of #1]}
  {\end{proof}}

\newcommand{\inner}[2]{\Braket{ #1, #2 }}
\usepackage{interval}
\intervalconfig{soft open fences}
\usepackage{xcolor}
\usepackage{braket}

\DeclareMathOperator{\dom}{dom}

\DeclareMathOperator*{\argmin}{ \arg \min }

\DeclareMathOperator{\KL}{KL}

\newcommand{\scaleparallel}[2]{
    \left. {#1} \middle| \kern-0.1em \middle| {#2} \right.
}

\newcommand{\RenyiKL}[2]{\KL\left( \scaleparallel{#1}{#2}\right)}

\author[1]{Chun-Neng~Chu}
\author[1,2,3]{Yen-Huan Li}

\affil[1]{Department of Computer Science and Information Engineering,\protect\\National Taiwan University}
\affil[2]{Department of Mathematics, National Taiwan University}
\affil[3]{Center for Quantum Science and Engineering, \protect\\ National Taiwan University}

\usepackage{etoolbox}
\newbool{versionArXiv}
\setbool{versionArXiv}{true} 

\begin{document}
    \ifbool{versionArXiv}{
\date{}
\title{Multiplicative Gradient Method for Log-Homogeneous Optimization: A Unified Mirror-Descent Framework}
}{
}
\maketitle
\begin{abstract}%
Consider the problems of computing $\ell_p$-Lewis weights, Kelly portfolios, D-optimal designs, general equilibria in Fisher markets, and optimal error exponents for channel coding in information theory. 
All of them can be formulated as optimization problems over the probability simplex, but except for the last one, the associated objective functions are neither Lipschitz continuous nor smooth, rendering existing first-order optimization algorithms and their guarantees not directly applicable. 
Nevertheless, first-order algorithms for solving these problems do exist in the literature and, perhaps surprisingly, look very similar, even though their interpretations, derivations, and analyses mostly appear unrelated. 
In this paper, we develop a framework unifying these results that captures their common problem structure through log-homogeneity and a novel second-order condition on the objective function, and identifies their common algorithmic pattern as multiplicative gradient methods. 

The framework recovers known algorithms and their optimization error guarantees for computing Kelly portfolios, D-optimal designs, optimal error exponents for channel coding in classical communication, $\ell_p$-Lewis weights with $p<2$, and general equilibria in Fisher markets.
To illustrate the insights provided by the framework, we further show that multiplicative gradient methods yield state-of-the-art performance guarantees  for computing $\ell_p$-Lewis weights with $p > 2$ and  optimal error exponents in classical-quantum communication.
Both results are novel to the best of our knowledge.

\end{abstract}

\ifbool{versionArXiv}{
}{
}

\section{Introduction}
\label{sec:Intro}


Consider the following three optimization problems.

\begin{enumerate}
\item The \emph{$\ell_p$-Lewis weight} for $p\in(0,2)\cup(2,\infty)$ for a matrix $A\in\mathbb{R}^{m\times n}$ is given by \citep{Lee2020}
\begin{align}
    \label{def:Lewis}
    x_{\text{Lew},\star}
    \in\argmin_{x\in \Delta_m}{f}_{\text{Lew}}(x),
    \quad
    {f}_{\text{Lew}}(x)
    \coloneqq\frac{p}{n\left(2-p\right)}\log\left(\det\left(M(x)\right)\right),
\end{align}
where $\Delta_m$ denotes the probability simplex in $\mathbb{R}^m$ and $M(x)\coloneqq\sum_{j=1}^m x[j]^{1-2/p}A[j]A[j]^{\top}$. 
Here, $x[j]$ denotes the $j^{\text{th}}$ coordinate of $x$, and $A[j]\in\mathbb{R}^n$ denotes the $j^{\text{th}}$ row of $A$, viewed as a column vector.
The notion of $\ell_p$-Lewis weights finds applications in $\ell_p$-regression \citep{Durfee2018,Parulekar2021}, linear programming \citep{Lee2020}, and log-concave sampling \citep{Kook2024,Jiang2024}.
\citet{Cohen2015b} proposed the following iteration rule:
\begin{align}
  \label{alg:CohenMGM}
  x_{t+1}[j] 
  =\frac{1}{n}\inner{M(nx_t)^{-1}A[j]}{A[j]}^{p/2},\quad\text{for all } 1\leq j\leq m.
\end{align}
They showed that the iterates converge to the minimizer at a rate of $\widetilde{O}(\abs{1-p/2}^T)$ for $p<4$, where $T$ is the number of iterations.

\item Consider a Fisher market with $m$ goods of unit supply and $n$ buyers, each with a budget and a constant-elasticity-of-substitution (CES) utility function.
The \emph{equilibrium prices} are given by \citep{Cheung2020}
\begin{align}
    \label{def:Fisher}
    x_{\text{Fis},\star}
    \in\argmin_{x\in \Delta_m}{f}_{\text{Fis}}(x),
    \quad
    {f}_{\text{Fis}}(x)
    \coloneqq\frac{1-\rho}{\rho}\sum_{i=1}^n w[i] \log\left(\inner{a_i}{x^{\rho/(\rho-1)}}\right),
\end{align}
where the power is applied entry-wise and $\rho\in(-\infty,0)\cup(0,1)$ is the CES parameter.
Here, $w[i]$ and the entrywise non-negative vector $a_i$ denote the $i^{\text{th}}$ buyer's budget and preference parameters over the $m$ goods, respectively.
In addition to their significance in general equilibrium theory, Fisher market models also find applications in recommender systems \citep{Kroer2019,Kroer2021Computing}, online advertising \citep{Conitzer2022,Bateni2021}, and edge computing \citep{Nguyen2019,Nguyen2021}.
For $\rho < 0$, \citet{Cheung2018c} proposed the following iteration rule:
\begin{align}
& b_{t+1}[i] = w[i]\inner{a_i}{x_t^{\rho/(\rho-1)}}^{-1}a_i\odot x_t^{\rho/(\rho-1)}\quad\text{for all } 1\leq i\leq n , \label{alg:IntroFisherMGM} \\
& x_{t + 1} = \sum_{k = 1}^n b_{t + 1} [ k ] , 
\end{align}
where $\odot$ denotes entry-wise multiplication and $b_t[i]$ represents how the $i^{\text{th}}$ buyer allocates their budget across the goods.
They showed that the optimization error converges to zero at a rate of $\widetilde{O}((\rho/(\rho-1))^T)$.

\item The \emph{Kelly criterion}, which yields an asymptotically optimal strategy for multi-round investment, requires solving an optimization problem of the following form at the beginning of each round \citep{Kelly1956,Breiman1961,Algoet1988}:
\begin{align}
    \label{def:Kelly}
    x_{\text{Kel},\star}
    \in\argmin_{x\in \Delta_m}{f}_{\text{Kel}}(x),
    \quad
    {f}_{\text{Kel}}(x)
    \coloneqq-\mathbb{E}_{a\sim P}\left[\log\left(\inner{a}{x}\right)\right],
\end{align}
where $m$ denotes the number of stocks, $a$ is an entry-wise nonnegative random vector whose entries are the price relatives of the stocks, and $P$ denotes the probability distribution of the price relatives conditional on the history.
Beyond mathematical finance, equivalent optimization problems also arise in positron emission tomography, astronomical imaging, and fluorescence microscopy \citep{Vardi1993,Bertero2009}.
\citet{Shepp1982} and \citet{Cover1984} proposed the following iteration rule:
\begin{align}
    \label{alg:KellyMGM}
    x_{t+1}[j]
    =x_t[j]\mathbb{E}_{a\sim P}\left[\frac{a[j]}{\inner{a}{x_t}}\right],\quad\text{for all } 1\leq j\leq m,
\end{align}
and proved its asymptotic convergence to the minimizer.
\citet{Iusem1992}, \citet{Lin2021}, and \citet{Frankowski2022} further proved that the 
optimization error
converges at a rate of $\widetilde{O}(1/T)$.
\end{enumerate}

Textbook first-order optimization methods, such as gradient descent and mirror descent, do not directly apply to the optimization problems \eqref{def:Lewis}, \eqref{def:Fisher}, and \eqref{def:Kelly} due to the lack of Lipschitz continuity and $L$-smoothness in their objective functions.
However, we identify a common pattern in the above algorithms:
the algorithm~\eqref{alg:CohenMGM} can be written as $x_{t+1}=x_t\odot(-\nabla f_{\text{Lew}}(x_t))^{p/2}$;
the algorithm~\eqref{alg:IntroFisherMGM} can be written as $x_{t+1}=x_t\odot(-\nabla f_{\text{Fis}}(x_t))$; 
the algorithm~\eqref{alg:KellyMGM} can be written as $x_{t+1}=x_t\odot(-\nabla f_{\text{Kel}}(x_t))$.
We refer to algorithms of the form $x_{t+1}\propto x_t\odot(-\nabla f(x_t))^{\eta}$ as the multiplicative gradient method (MGM).\footnote{In our analysis, we will impose normalization to ensure that iterates lie in the probability simplex.}
The MGM has also appeared in other applications, such as D-optimal experimental design \citep{Cohen2019,Silvey1978} and the computation of several information-theoretic quantities \citep{Arimoto1976,Tsai2024a}.

The MGM does not seem to fall under any standard optimization framework.
The connection among the above problems was previously unclear, and derivations and analyses of the MGM for these problems appear unrelated---to our knowledge, no single existing analysis applies to all these problems.
For example, 
\citet{Iusem1992} leveraged the specific structure of Kelly's portfolio problem~\eqref{def:Kelly} to analyze the iteration rule~\eqref{alg:KellyMGM}, but it is unclear how to extend his analysis to other problems.
\citet{Cohen2015b} and \citet{Chu2025} analyzed the MGM via a contraction mapping argument with respect to the Thompson metric \citep{Thompson1963}. 
The Thompson metric is defined only on the positive orthant, and hence their analysis does not apply to Kelly's portfolio problem, where the minimizer may contain zero entries.
\citet{Lee2016phdthesis} and \citet{Cohen2019} exploited the convexity of the log-gradient of the objective function, and \citet{Lin2021} extended their analysis approach to the quantum setting.
\citet{Zhao2026MGM} further extended this analysis to Euclidean Jordan algebras.
Unfortunately, as shown in Appendix~\ref{appendix:LogGradConvexCounterexample}, convexity of the log-gradient does not hold for the $\ell_p$-Lewis weight problem for $p<2$.
We will discuss other existing analysis approaches, showing that they do not directly apply to all applications mentioned above, in Section~\ref{sec:RelatedWork}.

To address this gap, we observe that the objective functions $f$ of all the above problems are $\theta$-log-homogeneous, meaning that for any scalar $r>0$, we have $f(r x) = f(x) + \theta \log r$.
Without loss of generality, we assume $\theta\in\Set{-1,1}$.
A crucial step in developing our framework is to identify the following second-order characterization, parametrized by $\theta$ and two additional scalars $\mu$ and $L$, that captures the common structure of all these objective functions:
\begin{align}
    \label{ineq:IntroHessianBounds}
    &\mu H_f(x)
    \leq\nabla^2 f(x)+\theta\nabla f(x)\nabla f(x)^{\top}
    \leq LH_f(x),\quad\text{for all } x\in\mathbb{R}_{++}^m,
\end{align}
where we require that, when $\theta=-1$, either $L\leq1$
or both $\mu\geq1$ and $L>1$.
Here, we define
\begin{align}
    H_f(x) 
    \coloneqq \mathrm{Diag}(\theta\nabla f(x)\odot x^{-1})-\nabla f(x)\nabla f(x)^{\top},
\end{align}
where $\mathrm{Diag}(v)$ denotes a diagonal matrix with the vector $v$ on its diagonal.
In our framework, we construct a surrogate objective function and a reference function and prove that the characterization~\eqref{ineq:IntroHessianBounds} implies $L$-relative smoothness and $\mu$-relative strong convexity with respect to the reference function.
This leads to a convergence rate for mirror descent of $\widetilde{O}((1-\mu/L)^T)$ for $\mu>0$ and $\widetilde{O}(1/T)$ for $\mu=0$ \citep{Bauschke2017,Lu2018b,Teboulle2018}.
Finally, we show that this mirror-descent method on the surrogate function is equivalent to the MGM on the original function $f$ and that similar convergence rates continue to hold for the original objective function.

\begin{table}[tb]
\centering
\fontsize{9}{11}\selectfont
\setlength{\tabcolsep}{2pt}
\begin{tabular*}{\linewidth}{@{\extracolsep{\fill}}
  p{\dimexpr0.24\linewidth-2\tabcolsep\relax}
  p{\dimexpr0.32\linewidth-2\tabcolsep\relax}
  p{\dimexpr0.14\linewidth-2\tabcolsep\relax}
  p{\dimexpr0.30\linewidth-2\tabcolsep\relax}
  @{}}
\hline
\noalign{\vskip 3pt}
\raggedright Problem
& \raggedright References
& \centering $\eta$
& \centering Optimization error bound
\tabularnewline

\noalign{\vskip 2pt}
\hline
\noalign{\vskip 2pt}

\raggedright $\ell_p$-Lewis weight~\eqref{def:Lewis}\newline $p\in(0,2)$
& \raggedright \citet{Cohen2015b}
& \centering $p/2$
& \centering $\widetilde O((1-p/2)^T)$
\tabularnewline

\raggedright 
& \raggedright \textbf{Our Results}
& \centering $p/2$
& \centering $\widetilde O((1-p/2)^T)$
\tabularnewline

\noalign{\vskip 2pt}
\hline
\noalign{\vskip 2pt}

\raggedright $\ell_p$-Lewis weight~\eqref{def:Lewis}\newline $p\in(2,4)$
& \raggedright \citet{Cohen2015b}
& \centering $p/2$
& \centering $\widetilde O((p/2-1)^T)$
\tabularnewline

\raggedright 
& \raggedright \citet{Lee2016phdthesis}
& \centering $1$
& \centering $\widetilde O(1/T)$
\tabularnewline


\raggedright 
& \raggedright \textbf{Our Results}
& \centering $1$
& \centering $\widetilde O((1-2/p)^T)$
\tabularnewline

\noalign{\vskip 2pt}
\hline
\noalign{\vskip 2pt}

\raggedright $\ell_p$-Lewis weight~\eqref{def:Lewis}\newline $p\geq 4$
& \raggedright \citet{Lee2016phdthesis}
& \centering $1$
& \centering $\widetilde O(1/T)$
\tabularnewline


\raggedright 
& \raggedright \textbf{Our Results}
& \centering $1$
& \centering $\widetilde O((1-2/p)^T)$
\tabularnewline

\noalign{\vskip 2pt}
\hline
\noalign{\vskip 2pt}

\raggedright market equilibrium~\eqref{def:Fisher}\newline $\rho\leq -1$
& \raggedright \citet{Cheung2018c}
& \centering $1$
& \centering $\widetilde O((\rho/(\rho-1))^T)$
\tabularnewline

\raggedright 
& \raggedright \textbf{Our Results}
& \centering $1$
& \centering $\widetilde O((\rho/(\rho-1))^T)$
\tabularnewline

\noalign{\vskip 2pt}
\hline
\noalign{\vskip 2pt}

\raggedright market equilibrium~\eqref{def:Fisher}\newline $\rho\in(-1,0)$
& \raggedright \citet{Cheung2018c}
& \centering $1$
& \centering $\widetilde O((\rho/(\rho-1))^T)$
\tabularnewline

\raggedright 
& \raggedright \citet{Chu2025}
& \centering $1-\rho$
& \centering $\widetilde O((-\rho)^T)$
\tabularnewline

\raggedright 
& \raggedright \textbf{Our Results}
& \centering $1$
& \centering $\widetilde O((\rho/(\rho-1))^T)$
\tabularnewline

\noalign{\vskip 2pt}
\hline
\noalign{\vskip 2pt}

\raggedright market equilibrium~\eqref{def:Fisher}\newline $\rho\in(0,1)$
& \raggedright \citet{Chu2025}
& \centering $1-\rho$
& \centering $\widetilde O\left(\rho^T\right)$
\tabularnewline

\raggedright 
& \raggedright \textbf{Our Results}
& \centering $1-\rho$
& \centering $\widetilde O(\rho^T)$
\tabularnewline

\noalign{\vskip 2pt}
\hline
\noalign{\vskip 2pt}

\raggedright Kelly's portfolio~\eqref{def:Kelly}
& \raggedright \citet{Iusem1992}, \citet{Lin2021}, and \citet{Frankowski2022}
& \centering $1$
& \centering $\widetilde O(1/T)$
\tabularnewline

\raggedright 
& \raggedright \textbf{Our Results}
& \centering $1$
& \centering $\widetilde O(1/T)$
\tabularnewline

\noalign{\vskip 2pt}
\hline
\noalign{\vskip 2pt}

\raggedright D-optimal design
& \raggedright \citet{Cohen2019}
& \centering $1$
& \centering $\widetilde O(1/T)$
\tabularnewline

\raggedright 
& \raggedright \textbf{Our Results}
& \centering $1$
& \centering $\widetilde O(1/T)$
\tabularnewline

\noalign{\vskip 2pt}
\hline
\noalign{\vskip 2pt}

\raggedright R\'enyi capacity
& \raggedright \citet{Arimoto1976}
& \centering $\alpha/(1-\alpha)$
& \centering $\widetilde O(1/T)$
\tabularnewline

\raggedright 
& \raggedright \textbf{Our Results}
& \centering $\alpha/(1-\alpha)$
& \centering $\widetilde O(1/T)$
\tabularnewline

\noalign{\vskip 2pt}
\hline
\noalign{\vskip 2pt}

\raggedright Petz-R\'enyi capacity
& \raggedright \textbf{Our Results}
& \centering $\alpha/(1-\alpha)$
& \centering $\widetilde O(1/T)$
\tabularnewline

\noalign{\vskip 2pt}\hline
\end{tabular*}
\caption{
    Optimization error bound of the MGM after $T$ iterations.
    The parameter $\eta$ appears in the exponent of the gradient term in the MGM update.
}
\label{tab:MGMComparison}
\end{table}

In Section~\ref{sec:Applications}, we apply  our framework to various problems.
Table~\ref{tab:MGMComparison} summarizes our results and compares them with prior work.
We recover known algorithms and their optimization error bounds\footnote{For the market equilibrium and $\ell_p$-Lewis weight problems, some prior work also obtained convergence guarantees in terms of a different error measure. 
A discussion can be found in Appendix~\ref{appendix:ComparisonErrorMeasures}.} for the problems of computing $\ell_p$-Lewis weights~\eqref{def:Lewis} with $p<2$, market equilibria~\eqref{def:Fisher}, Kelly's portfolios~\eqref{def:Kelly}, D-optimal designs, and R\'{e}nyi capacities---characterizing the optimal error exponent for channel coding in classical communication.

Furthermore, to illustrate the insights provided by our framework, in Section~\ref{sec:Applications}, we show that the MGM yields state-of-the-art performance for computing $\ell_p$-Lewis weights with $p>2$ and Petz-R\'{e}nyi capacities, which characterize the optimal error exponent for channel coding in classical-quantum communication---novel results for the MGM to our knowledge.\footnote{We are aware of the concurrent work of \citet{Padmanabhan2026}, which establishes the same convergence guarantee as ours for the MGM for computing $\ell_p$-Lewis weights. A discussion can be found in Remark~\ref{remark:ConcurrentWorksLewisWeights}.}

Our framework is primarily inspired by \citet{Lee2020}. 
They established the second-order characterization~\eqref{ineq:IntroHessianBounds} specifically for the Lewis weight problem~\eqref{def:Lewis}, although the algorithm they analyzed is not an instance of the MGM.
This characterization is analogous to the standard Hessian bound, $\mu I_m\leq\nabla^2 f(x)\leq LI_m$, for $L$-smoothness and $\mu$-strong convexity, where additional  rank-one terms are introduced and the identity matrix $I_m$ is replaced with a diagonal matrix that depends on both the reference point $x$ and the objective function itself.
We verify that all the aforementioned applications satisfy the characterization~\eqref{ineq:IntroHessianBounds}, which motivates us to use it to develop a unified framework for the MGM.

\section{Related Work}
\label{sec:RelatedWork}
Although the MGM has been studied in various contexts, its derivations and formulations are often quite different:
the algorithm~\eqref{alg:CohenMGM} for the $\ell_p$-Lewis weight problem~\eqref{def:Lewis} and the algorithm~\eqref{alg:KellyMGM} for Kelly's portfolio problem~\eqref{def:Kelly} were derived by \citet{Cohen2015b} and \citet{Cover1984}, respectively, as fixed-point iterations;
the same algorithm~\eqref{alg:KellyMGM} was derived by \citet{Shepp1982} as an instance of expectation maximization, and by \citet{Csiszar1984} using alternating I-projection;
the algorithm~\eqref{alg:IntroFisherMGM} for the equilibrium problem~\eqref{def:Fisher} was derived by \citet{Cheung2018c} as an instance of mirror descent for an equivalent optimization problem.
Our framework unifies these algorithms under a general MGM expression that admits a principled convergence analysis.

The work of \citet{Cheung2018c} and \citet{Frankowski2022} on the equilibrium problem~\eqref{def:Fisher} and Kelly's portfolio problem~\eqref{def:Kelly} is perhaps the closest to our framework.
\citet{Cheung2018c} derived the algorithm~\eqref{alg:IntroFisherMGM} by applying mirror descent to minimize a surrogate function of the original objective function, which, as discussed in Section~\ref{sec:Intro}, is equivalent to applying the MGM to the original problem.
Similarly, \citet{Frankowski2022} recovered the algorithm~\eqref{alg:KellyMGM} by applying mirror descent to a surrogate  function.
While these works also provide an interpretation of the MGM as mirror descent tailored to the two specific problems, their surrogate functions require lifting the original dimension from $m$ to $mn$ and to potentially infinite dimensions, respectively,  and it is unclear whether these lifting mechanisms extend to other problems.
In contrast, our framework performs mirror descent on a different surrogate function without lifting the input dimension.

Mirror descent has known connections to natural gradient descent \citep{Raskutti2015}, alternating minimization \citep{Leger2023}, dual preconditioned methods \citep{Maddison2021,Laude2025,Oikonomidis2025GeneralSmooth}, and difference-of-convex methods \citep{Faust2023,Oikonomidis2025Forward}.
Through these established connections, our mirror-descent interpretation also links the MGM to these broader frameworks, which may be of independent interest.

The optimal error exponent in channel coding is characterized by the \emph{R\'{e}nyi capacity}. 
Interestingly, for the problem of computing the R\'{e}nyi capacity, the operational meanings of the objective function and associated surrogate function in our framework have been well established in the literature \citep{Gallager1965,Augustin1978,Csiszar1995,Nakiboglu2019}.
However, the connection between minimizing the two functions from an algorithmic perspective remains unexplored.
The algorithm of \citet{Arimoto1976} for computing the R\'{e}nyi capacity can be viewed as the MGM for minimizing $f$, but its analysis does not adopt the surrogate function perspective. 

While the R\'{e}nyi capacity admits an extension, called the Petz-R\'{e}nyi capacity, to classical-quantum communication \citep{Cheng2022}, the analysis of \citet{Arimoto1976}  does not admit a direct extension to this case, where the non-commutativity inherent in the mathematical formalism of quantum mechanics poses significant challenges.
In contrast, our framework directly applies to this quantum extension without additional effort.

\section{Preliminaries}
\label{sec:Preliminaries}
\subsection{Notation}
We denote the nonnegative and positive orthants in $\mathbb{R}^m$  by $\mathbb{R}_{+}^m$ and $\mathbb{R}_{++}^m$, respectively.
We denote the set of all Hermitian matrices in $\mathbb{C}^{m\times m}$ by $\mathcal{H}_m$.
We denote the positive semidefinite and positive definite cones of $\mathcal{H}_m$ by $\mathcal{H}_{m,+}$ and $\mathcal{H}_{m,++}$, respectively.
We denote the set of all Hermitian positive semidefinite matrices with unit trace in $\mathbb{C}^{m\times m}$ by $\mathcal{D}_m$.
We denote by $\mathbf{1}_m$ the all-ones vector in $\mathbb{R}^m$.
We denote by $I_m$ the identity matrix in $\mathbb{C}^{m\times m}$.

For any set $\mathcal{S}$, we denote by $\mathrm{ri}(\mathcal{S})$ its relative interior.
For any function $f\colon \mathbb{R}^m \to (-\infty,\infty]$, we denote its subdifferential and convex conjugate by $\partial f$ and $f^{*}$, respectively.
For any function $f\colon \mathbb{R} \to \mathbb{R}$ and 
vector \( v \in \mathbb{R}^m \), we define $f(v)$
as 
the 
\( m \)-dimensional vector $(f(v[1]), \dots, f(v[m]))$. 
Similarly, for such a function \( f \) and \( Q \in \mathcal{H}_m \), we define \( f(Q) \) as \( \sum_{i=1}^m f(\lambda_i)  P_i \), where \( Q = \sum_{i=1}^m \lambda_i   P_i \) is the eigendecomposition of \( Q \). 
We denote by $\oslash$ entry-wise division.
We denote by $D_{f}(y,x)=f(y)-f(x)-\inner{\nabla f(x)}{y-x}$ and $D_{g}(Y,X)=g(Y)-g(X)-\mathrm{Tr}\left(\nabla g(X)(Y-X)\right)$ the Bregman divergences induced by functions $f\colon\mathbb{R}^m\to(-\infty,\infty]$ and $g\colon\mathcal{H}_m\to(-\infty,\infty]$ between $x$ and $y$, and between $X$ and $Y$, respectively.
We denote by $\RenyiKL{y}{x}=\inner{y}{\log(y)-\log(x)}$ the KL divergence  between $x,y\in\Delta_m$.

\subsection{Mirror Descent}
Mirror descent is a standard optimization method for solving constrained convex optimization problems $\tilde{q}_{\star}\in\argmin_{\tilde{q}\in\tilde{\mathcal{C}}} \tilde{G}(\tilde{q})$.
In particular, it requires first choosing a reference function $\tilde{R}$ and a step size $\eta>0$, and then updating the iterates by
\begin{align}
    \label{alg:GeneralMD}
    \tilde{q}_{t+1}
    \in \argmin_{q\in\tilde{\mathcal{C}}} \eta\inner{\nabla \tilde{G}(\tilde{q}_t)}{q} + D_{\tilde{R}}(q,\tilde{q}_t).
\end{align}

Relative smoothness and relative strong convexity are two standard assumptions for the convergence analysis of mirror descent \citep{Bauschke2017,Lu2018b,Teboulle2018}.
On the constraint set $\tilde{\mathcal{C}}$, we say a function $\tilde{G}$ is $L$-relatively smooth with respect to $\tilde{R}$ if the function $L\tilde{R}-\tilde{G}$ is convex, and $\mu$-relatively strongly convex with respect to $\tilde{R}$ if the function $\tilde{G}-\mu\tilde{R}$ is convex.

Existing convergence analyses typically assume that the reference function is differentiable.
Although \citet{Lu2018b} allow both the reference function $\tilde{R}$ and the objective function $\tilde{G}$ to be non-differentiable, they require $\tilde{R}-\tilde{G}$ to be differentiable.
While our setting may not satisfy these differentiability assumptions, we can still establish convergence guarantees by slightly extending the analysis of \citet{Teboulle2018} to handle subgradients in place of gradients.

\section{Main Results}
\label{sec:MainResults}
This section presents our framework and establishes
convergence guarantees for the MGM.

We consider the following optimization problem:
\begin{align}
    \label{def:GeneralProblem}
    x_{\star}\in\argmin_{x\in\Delta_m} f(x),
\end{align}
where $f$ satisfies Assumption~\ref{assumption:BasicAssumptions} below and the  second-order characterization~\eqref{ineq:IntroHessianBounds} for some $\mu,L$ with $0\leq\mu\leq L$ and $L>0$.
\begin{assumption}
    \label{assumption:BasicAssumptions}
    The function $f\colon\mathbb{R}^m\to(-\infty,\infty]$ is continuous on $\mathbb{R}_{+}^m\setminus\Set{0}$, is also $\theta$-log-homogeneous for some $\theta\in\Set{-1,1}$, and satisfies $\theta\nabla f(x)>0$ for all $x\in\mathbb{R}_{++}^m$.
\end{assumption}


    

We consider solving the optimization problem~\eqref{def:GeneralProblem} by the multiplicative gradient method (MGM):
\begin{align}
    \label{alg:IntroMGM}
    x_{t+1}
    =\frac{x_t\odot(\theta\nabla f(x_t))^{-\theta\eta}}{\norm{x_t\odot(\theta\nabla f(x_t))^{-\theta\eta}}_1}.
\end{align}

Our convergence analysis is based on viewing the MGM as mirror descent with a reference function applied to minimizing a surrogate function.
The construction of the surrogate function and reference function will rely on the following transformation of the original objective function $f$:
\begin{align}
    \label{def:varphi}
    \varphi(y)\coloneqq c_f f(\exp(c_f\theta y)),
\end{align}
where we define the scalar $c_f$ as $-1$ when $\theta=-1$ with $0\leq \mu\leq L\leq 1$, and as $1$ when either $\theta=-1$ with $1\leq \mu\leq L$ and $L>1$, or $\theta=1$.
\begin{lemma}
    \label{lemma:ConjugateFunctionWellDefine}
    The function $\varphi(y)$ is convex.
\end{lemma}
The proof of Lemma~\ref{lemma:ConjugateFunctionWellDefine} is deferred to Appendix~\ref{appendix:ProofConjugateFunctionWellDefine}.

We apply mirror descent with a reference function to the following optimization problem:
\begin{align}
    \label{def:ConjugateProblem}
    q_{\star}\in\argmin_{q\in\mathcal{C}} G^{f}(q).
\end{align}
Here, the associated surrogate function is defined by
$G^f(q)\coloneqq -\theta\varphi^{*}(q)+c_f\inner{q}{\log(q)}$
for $q\in\mathcal{C}$ and by $G^f(q)\coloneqq+\infty$ otherwise,
where $\mathcal{C}\coloneqq\dom(\varphi^{*})\cap\Delta_m$.
The reference function is defined as the convex conjugate of $\varphi$.

The following lemma is one of our key technical results in this section, establishing  the relative smoothness and relative strong convexity properties of the associated surrogate function $G^f$ with respect to the reference function $\varphi^*$.
\begin{lemma}
    \label{lemma:RelativeSmoothnessStrongConvexity}
    Suppose that Assumption~\ref{assumption:BasicAssumptions} and the second-order characterization~\eqref{ineq:IntroHessianBounds} are satisfied by the objective function $f$, and that either $\theta=-1$ with $0\leq \mu\leq L\leq 1$ and $L>0$, $\theta=-1$ with $1\leq \mu\leq L$ and $L>1$, or $\theta=1$ with $L>0$ and $0\leq\mu\leq L$.
    Then, the associated surrogate function $G^{f}$ is $L$-smooth and $\mu$-strongly convex relative to the reference function $\varphi^{*}$ on  $\mathcal{C}$, where $\mathcal{C}\coloneqq\dom(\varphi^{*})\cap\Delta_m$.
\end{lemma}
The proof of Lemma~\ref{lemma:RelativeSmoothnessStrongConvexity} is deferred to Appendix~\ref{appendix:ProofRelativeSmoothnessStrongConvexity}.

A challenge arises in implementing mirror descent: 
the function $\varphi^*$ may not have closed-form expressions at all points.
To address this issue, we note that, by convex conjugacy, we have
\begin{align}
    \label{eq:Conjugate}
    &\varphi^{*}(\nabla \varphi(c_f\theta \log(x)))
    =\inner{c_f\theta \log(x)}{\nabla \varphi(c_f\theta \log(x))}-c_f f(x),\nonumber\\
    &c_f\theta \log(x)\in\partial \varphi^{*}(\nabla \varphi(c_f\theta \log(x))),\quad\text{for all } x\in\mathbb{R}_{++}^m.
\end{align}
This allows us to perform mirror descent with closed-form updates, provided that the iterates can be represented in the form  $\nabla \varphi(c_f\theta \log(x))$.
The following lemma, another key result in this section, shows that the MGM iterates yield such a representation, thereby not only facilitating the implementation of mirror descent but also establishing its connection to the MGM.
\begin{lemma}
    \label{lemma:MGMEqConjugateMD}
    Let the function $f$ and the set $\mathcal{C}$ be as in Lemma~\ref{lemma:RelativeSmoothnessStrongConvexity},
    and let $x_t$ be the iterates generated by the MGM~\eqref{alg:IntroMGM} with $\eta=1/L$ and the initial point $x_1\in\mathrm{ri}(\Delta_m)$.
    Furthermore, let $q_t\coloneqq\nabla \varphi(c_f\theta \log(x_t))$ for all $t$.
    Then, for all $t$, we have $q_t\in\mathcal{C}$, $r_t\in\partial \varphi^{*}(q_t)$, and $g_t\in\partial G^{f}(q_t)$, where
    \begin{align}
        r_t
        \coloneqq c_f\theta\log(x_t),\quad g_t\coloneqq -c_f\log(x_t)+c_f\log(q_t)+c_f\mathbf{1}_m.
    \end{align}
    Moreover, with the subgradient choices $\nabla\varphi^{*}(q_t)=r_t$ and $\nabla G^{f}(q_t)=g_t$, the sequence $q_t$ follows mirror descent~\eqref{alg:GeneralMD} with $\tilde{G}=G^f$, $\tilde{R}=\varphi^{*}$, $\tilde{\mathcal{C}}=\mathcal{C}$, $\tilde{q}_t=q_t$, and the same step size $\eta=1/L$.
\end{lemma}
The proof of Lemma~\ref{lemma:MGMEqConjugateMD} is deferred to Appendix~\ref{appendix:ProofMGMEqConjugateMD}.

Leveraging Lemma~\ref{lemma:RelativeSmoothnessStrongConvexity} and Lemma~\ref{lemma:MGMEqConjugateMD}, and following the analysis of \citet{Teboulle2018} with gradients replaced by subgradients,
we obtain the following convergence guarantee.
\begin{corollary}
    \label{corollary:ConjugateMDConv}
    Let $f$ be as in Lemma~\ref{lemma:RelativeSmoothnessStrongConvexity},
    and let $q_t$ be defined as in Lemma~\ref{lemma:MGMEqConjugateMD}.
    Then, we have
    \begin{align}
        G^f(q_{T+1})-G^f(q_{\star})
        \leq L\min\Set{\frac{1}{T},{\left(1-\frac{\mu}{L}\right)^T}}D_{\varphi^{*}}(q_{\star},q_1),
    \end{align}
    where  the subgradient $\nabla\varphi^{*}(q_1)$ in the Bregman divergence is defined as in Lemma~\ref{lemma:MGMEqConjugateMD}.
\end{corollary}
The proof of Corollary~\ref{corollary:ConjugateMDConv} is deferred to Appendix~\ref{appendix:ProofConjugateMDConv}.

Finally, to obtain convergence guarantees for the original objective function $f$, we
establish below a connection between optimization problems \eqref{def:GeneralProblem} and \eqref{def:ConjugateProblem}.
\begin{lemma}
    \label{lemma:EquivalenceOpt}
    Let $f$ be as in Lemma~\ref{lemma:RelativeSmoothnessStrongConvexity}.
    Then, all of the minimizers $x_{\star}$ of the optimization problem~\eqref{def:GeneralProblem} are also minimizers of the optimization problem~\eqref{def:ConjugateProblem}.
\end{lemma}
The proof of Lemma~\ref{lemma:EquivalenceOpt} is deferred to Appendix~\ref{appendix:ProofEquivalenceOpt}.

Then, we use the identity~\eqref{eq:Conjugate} from the convex conjugacy to express $G^f(q_t)$ and $\varphi^{*}(q_t)$ in terms of the original objective function  $f$ and the iterate $x_t$, and then exploit the relative smoothness and relative strong convexity of $G^f$ with respect to $\varphi^{*}$ to obtain the following convergence guarantee.
\begin{theorem}
    \label{theorem:ConvergenceRateMGM}
    Let $f$ be as in Lemma~\ref{lemma:RelativeSmoothnessStrongConvexity},
    and let $x_t$ be the iterates generated by the MGM~\eqref{alg:IntroMGM} with $\eta=1/L$ and $x_1\in\mathrm{ri}(\Delta_m)$.
    Then, we have
    \begin{align}
        &f(x_{T+1})-f(x_{\star})\nonumber\\
        &\quad\leq\begin{cases}
            \min\Set{\frac{1}{T},{\left(1-\frac{\mu}{L}\right)^T}}L\left(K_1-O_1\right),&\quad{\text{if }}\theta=-1\text{ and }0\leq \mu\leq L\leq 1,\\
            \left(1-\frac{\mu}{L}\right)^T O_1,&\quad{\text{if }}\theta=-1,\text{ } 1\leq \mu\leq L,\text{ and }L>1,\\
            \min\Set{\frac{1}{T},{\left(1-\frac{\mu}{L}\right)^T}}L\left(K_1+O_1\right),&\quad{\text{if }}\theta=1,\text{ }0\leq \mu\leq L, \text{ and }L>0,
        \end{cases}
    \end{align}
    where $O_1= f(x_1)-f(x_{\star})$ and $K_1=\RenyiKL{x_{\star}}{x_1}$.
\end{theorem}
The proof of Theorem~\ref{theorem:ConvergenceRateMGM} is deferred to Appendix~\ref{appendix:ProofConvergenceRateMGM}.

\section{Applications}
\label{sec:Applications}
In Appendix~\ref{appendix:Applications}, 
we show that the objective functions of the market equilibrium problem~\eqref{def:Fisher}, Kelly's portfolio problem~\eqref{def:Kelly}, and the D-optimal design problem satisfy Assumption~\ref{assumption:BasicAssumptions} and the second-order condition~\eqref{ineq:IntroHessianBounds}.
By setting $\eta=1/L$ in the MGM~\eqref{alg:IntroMGM} and applying the convergence guarantee established in Theorem~\ref{theorem:ConvergenceRateMGM}, we recover known algorithms and their optimization error guarantees, as summarized in Table~\ref{tab:MGMComparison}.
Here, $L$ is the parameter in the second-order condition~\eqref{ineq:IntroHessianBounds}.

In Section~\ref{subsec:LewisWeights} and Section~\ref{subsec:PetzRenyiCapacity}, we apply our framework to the problems of computing $\ell_p$-Lewis weights~\eqref{def:Lewis} and R\'{e}nyi capacities in both classical and classical-quantum communication.
We recover known algorithms and their optimization error guarantees for computing  $\ell_p$-Lewis weights  with $p<2$ and R\'{e}nyi capacities in classical communication.
Furthermore, we show that the MGM yields state-of-the-art performance guarantees for computing $\ell_p$-Lewis weights with $p>2$ and  R\'{e}nyi capacities in classical-quantum communication, i.e., Petz-R\'{e}nyi capacities.
Both of these results are novel to the best of our knowledge.



\subsection{Application to Lewis Weight}
\label{subsec:LewisWeights}


Consider the problem of computing the $\ell_p$-Lewis weight~\eqref{def:Lewis}. 
Without loss of generality, we assume that the matrix $A$ has full column rank and has no zero rows.
If this assumption does not hold, we may remove zero rows and replace $A$ by a full-column-rank matrix with the same column space.

\begin{lemma}[{[\citealp{Cohen2015b}, Fact~2.1(2)]}]
    \label{lemma:LewisSolutionUpperBound}
    All entries of the optimal solution $x_{\text{Lew},\star}$ to the optimization problem~\eqref{def:Lewis} are no larger than $1/n$.
\end{lemma}
We assume that $n<m$; otherwise, Lemma~\ref{lemma:LewisSolutionUpperBound} directly implies that $x_{\text{Lew},\star}=\mathbf{1}_n/n$.

To establish the second-order characterization~\eqref{ineq:IntroHessianBounds}, we use  the following known Hessian bound. 
Indeed, as discussed in Section~\ref{sec:Intro}, it motivates our second-order characterization.
\begin{lemma}[{[\citealp{Lee2020}, Lemma~23]}]    
    \label{lemma:KnownHessianBoundsLewis}
    For any $x\in\mathbb{R}_{++}^m$, the objective function ${f}_{\text{Lew}}$ of the optimization problem~\eqref{def:Lewis} satisfies
    \begin{align}
        \min\Set{\frac{2}{p},1}\mathrm{Diag}(-\nabla {f}_{\text{Lew}}(x)\oslash x)
        \leq \nabla^2 {f}_{\text{Lew}}(x)
        \leq \max\Set{\frac{2}{p},1}\mathrm{Diag}(-\nabla {f}_{\text{Lew}}(x)\oslash x).
    \end{align}
\end{lemma}
By using the log-homogeneity of the function ${f}_{\text{Lew}}(x)$, we can add rank-one matrices to the Hessian bounds in Lemma~\ref{lemma:KnownHessianBoundsLewis} so that they align with the second-order characterization~\eqref{ineq:IntroHessianBounds}.
Then, we obtain the following lemma.
\begin{lemma}
    \label{lemma:HessianBoundsLewis}
    The objective function ${f}_{\text{Lew}}$  satisfies Assumption~\ref{assumption:BasicAssumptions} and the second-order characterization~\eqref{ineq:IntroHessianBounds} with $(\theta,\mu,L)=(-1,\min\Set{2/p,1},\max\Set{2/p,1})$.
\end{lemma}
The proof of Lemma~\ref{lemma:HessianBoundsLewis} is deferred to Appendix~\ref{appendix:ProofHessianBoundsLewis}.

Combining Lemma~\ref{lemma:LewisSolutionUpperBound}
with Lemma~\ref{lemma:HessianBoundsLewis} and Theorem~\ref{theorem:ConvergenceRateMGM}, we obtain the following convergence result for $p\in(0,2)\cup(2,\infty)$:
\begin{corollary}
    \label{cor:LewisMGMOptErrorKLConv}
    Let $x_t$ be the iterates generated by the MGM~\eqref{alg:IntroMGM} for solving the optimization problem~\eqref{def:Lewis} with  $\eta=1/\max\Set{2/p,1}$, $\theta=-1$, and $x_1=\mathbf{1}_m/m$.
    Then, for all $p\in(0,2)\cup(2,\infty)$, we have
    \begin{align}
        f_{\text{Lew}}(x_{T+1})-f_{\text{Lew}}(x_{\text{Lew},\star})
        \leq \left(1-\min\Set{\frac{2}{p},\frac{p}{2}}\right)^T\log\left(\frac{m}{n}\right).\nonumber
    \end{align}
\end{corollary}

As for the case where $p>2$, we characterize below the iteration complexity for the MGM to obtain an $\varepsilon$-estimate of the $\ell_p$-Lewis weight.
Here, we say a vector $x$ is an $\varepsilon$-estimate of the $\ell_p$-Lewis weight if the multiplicative error $\norm{x\oslash x_{\text{Lew},\star}-\mathbf{1}_m}_{\infty}$ is at most $\varepsilon$.

\begin{corollary}
    \label{cor:MGMMultiplicativeEstimateConv}
    Let $x_t$ be the iterates generated by the MGM~\eqref{alg:IntroMGM} with $\eta,\theta$ and $x_1$ chosen as in Corollary~\ref{cor:LewisMGMOptErrorKLConv}.
    Then, for $p>2$, after $T$ iterations, where
    \begin{align}
        T
        \geq\frac{p}{2}\log\left(16\frac{np^2}{\varepsilon^2}\log\left(\frac{m}{n}\right)\right)+1,\nonumber
    \end{align}
    the vector $x_{T+1} \odot (-\nabla {f}_{\text{Lew}}(x_{T+1}))^{p/2}$ is an $\varepsilon$-estimate of the $\ell_p$-Lewis weight for all $\varepsilon\in(0,1]$.
\end{corollary}
The proof of Corollary~\ref{cor:MGMMultiplicativeEstimateConv} is deferred to Appendix~\ref{appendix:ProofMGMMultiplicativeEstimateConv}.
Indeed, \citet{Gribling2026} have shown that an $\varepsilon$-estimate
of the $\ell_p$-Lewis weight can be constructed from a vector $x$
provided that the matrix $M(x)$ is sufficiently close to
$M(x_{\text{Lew},\star})$.
We further show that such closeness can be guaranteed when the optimization error is sufficiently small, thereby allowing us to translate the optimization error bound in Corollary~\ref{cor:LewisMGMOptErrorKLConv} into an iteration complexity bound for obtaining an $\varepsilon$-estimate of the $\ell_p$-Lewis weight using the MGM.





\subsubsection{Comparison with Existing Results}
\label{subsubsec:LewisComparison}

For computing an $\varepsilon$-estimate of the $\ell_p$-Lewis weight~\eqref{def:Lewis} for $p>2$, the algorithm of \citet{Lee2016phdthesis} can also be equivalently written as the MGM~\eqref{alg:IntroMGM} with step size $\eta = 1$. 
However, they only established an iteration complexity of $\widetilde{O}(1/\varepsilon)$, which is higher than the complexities $\widetilde{O}(p^{3}\log(1/\varepsilon))$ and $\widetilde{O}(p^{2}\log(1/\varepsilon))$ achieved by the algorithms proposed by \citet{Fazel2022} and \citet{Gribling2026}, respectively.
Our framework recovers the MGM with step size $\eta = 1$ chosen based on the second-order characterization~\eqref{ineq:IntroHessianBounds} and
yields an improved complexity guarantee of $\widetilde{O}(p\log(1/\varepsilon))$, which is state-of-the-art among all previous algorithms.

As discussed in Section~\ref{sec:Intro}, the algorithm of \citet{Cohen2015b} can be viewed as the MGM~\eqref{alg:IntroMGM} with step size $\eta=p/2$.
They proved that the iterates converge to the optimal point at a rate of $\tilde{O}(\abs{1-p/2}^T)$ for $p<4$ in terms of the Thompson metric $\norm{\log(x)-\log(x_{\text{Lew},\star})}_{\infty}$.
As shown in Appendix~\ref{appendix:ComparisonErrorMeasures}, this metric upper bounds the optimization error, and thus their result also implies the same convergence rate in terms of the optimization error bound.
Note that while our convergence rate $\tilde{O}((1-\min\Set{2/p,p/2})^T)$  established in Corollary~\ref{cor:LewisMGMOptErrorKLConv} dominates their result in terms of the optimization error bound, we do not recover their convergence guarantee in terms of the Thompson metric.

\begin{remark}
    \label{remark:ConcurrentWorksLewisWeights}
    After completing this work, we became aware of the following two concurrent works:
    \citet{Padmanabhan2026} established the same complexity of $\widetilde{O}(p\log(1/\varepsilon))$ for computing the $\ell_p$-Lewis weight using the MGM~\eqref{alg:IntroMGM} with step size $\eta = 1$ for $p>2$; the algorithm of \citet{Song2026Lewis} achieves a lower complexity of $\widetilde{O}(p\log(p)+\sqrt{p}\log(1/\varepsilon))$ for $p\geq 4$.
    We note that their analyses are specific to the structure of the $\ell_p$-Lewis weight problem.
\end{remark}

\subsection{Application to Petz-R\'{e}nyi Capacity}
\label{subsec:PetzRenyiCapacity}
The Petz-R\'{e}nyi capacity extends the R\'{e}nyi capacity to the quantum setting and characterizes the optimal error exponent in communication over classical-quantum channels \citep{Dalai2017,Cheng2019,Renes2025,Li2025b,Cheng2026}.
Here, the error exponent stands for the rate at which the decoding error probability decays exponentially with the length of the code over the channel.
A classical-quantum channel is a mapping $\mathcal{W}$ from
the set of \emph{classical symbols} $\Set{1,2,\dots,m}$ to
the set $\mathcal{D}_n$ of $n\times n$ \emph{density matrices}, which are Hermitian positive semidefinite matrices with unit trace.
In the mathematical formalism of quantum mechanics, a density matrix corresponds to a quantum state. 
The  Petz-R\'{e}nyi capacity $C_{\mathcal{W},\alpha}$ of order $\alpha\in(0,1)$ is given by
\begin{align}
    \label{def:PetzRenyiCapacity}
    C_{\mathcal{W},\alpha}
    \coloneqq\frac{1}{\alpha-1}\min_{x\in\Delta_m} f_{\text{Petz}}(x),
    \quad f_{\text{Petz}}(x)
    \coloneqq\alpha\log \left(\mathrm{Tr}\left(W(x)^{1/\alpha}\right)\right),
\end{align}
where $W(x)\coloneqq\sum_{j=1}^m x[j] \mathcal{W}(j)^{\alpha}$.

The lemma below is our main technical contribution in this section.
\begin{lemma}
    \label{lemma:HessianBoundsPetz}
    The objective function $f_{\text{Petz}}$ in the optimization problem~\eqref{def:PetzRenyiCapacity} satisfies Assumption~\ref{assumption:BasicAssumptions} and the second-order characterization~\eqref{ineq:IntroHessianBounds} with $(\theta,\mu,L)=(1,0,1/\alpha-1)$ for all $\alpha\in(0,1)$.
\end{lemma}
The proof of Lemma~\ref{lemma:HessianBoundsPetz} is deferred to Appendix~\ref{appendix:ProofHessianBoundsPetz}.

Combining Lemma~\ref{lemma:HessianBoundsPetz} and  Theorem~\ref{theorem:ConvergenceRateMGM}, we obtain the following corollary.

\begin{corollary}
    \label{cor:PetzMGMOptErrorKLConv}
    Let $x_t$ be the iterates generated by the MGM~\eqref{alg:IntroMGM} for solving the optimization problem~\eqref{def:PetzRenyiCapacity}, with $\eta=\alpha/(1-\alpha)$, $\theta=1$, and $x_1=\mathbf{1}_m/m$.
    Then, for all $\alpha\in(0,1)$, we have
    \begin{align}
        0
        \leq C_{\mathcal{W},\alpha}-\frac{1}{\alpha-1}f_{\text{Petz}}(x_{T+1})
        \leq \frac{(2-\alpha)\log(m)}{\alpha T}.
    \end{align}
\end{corollary}
The proof of Corollary~\ref{cor:PetzMGMOptErrorKLConv} is deferred to Appendix~\ref{appendix:ProofPetzMGMOptErrorKLConv}.

\subsubsection{Comparison with Existing Methods}
\label{subsubsec:PetzComparison}
The R\'{e}nyi capacities can be seen as special cases of the Petz-R\'{e}nyi capacities~\eqref{def:PetzRenyiCapacity} when all the involved matrices share the same eigenbasis.
In this special case, as discussed in Section~\ref{sec:RelatedWork}, the algorithm of \citet{Arimoto1976} can be viewed as an instance of the MGM~\eqref{alg:IntroMGM} with $\eta=\alpha/(1-\alpha)$ for computing R\'{e}nyi capacities.
He established a convergence rate of $O(\log(m)/(\alpha T))$.
As shown in Corollary~\ref{cor:PetzMGMOptErrorKLConv}, our framework recovers his algorithm and its convergence rate.
Furthermore, as discussed in Section~\ref{sec:RelatedWork}, the analysis of \citet{Arimoto1976} does not admit a direct extension to the general case where the matrices may not share the same eigenbasis, i.e., the non-commutative case.
In contrast, our framework directly applies to the general non-commutative case.

The algorithms of \citet{Lai2026} and \citet{Chu2026} for computing the Petz-R\'{e}nyi capacity $C_{\mathcal{W},\alpha}$ achieve convergence rates of $O(C\log(m)T^{-1})$ and $O(C\log(m)^{0.5/\alpha}T^{1-1.5/\alpha})$, for $\alpha\in(0,0.5]$ and $\alpha\in[0.5,1)$, respectively.
These rates scale exponentially with the Petz-R\'{e}nyi capacity $C_{\mathcal{W},\alpha}$, i.e., the factor $C$ is given by $\exp((1/\alpha-1)C_{\mathcal{W},\alpha})$.
In contrast, our method achieves a convergence rate of $O((2-\alpha)\log(m)/(\alpha T))$ in Corollary~\ref{cor:PetzMGMOptErrorKLConv} for all $\alpha\in(0,1)$, which does not involve such a factor.

The algorithm of \citet{Chu2025} achieves a convergence rate similar to that of our method for all $\alpha\in(0.5,1)$.
However, the per-iteration time complexity of their algorithm is higher, as it requires solving a subproblem in each iteration that does not have a closed-form solution.
In contrast, each iteration of our MGM admits a closed-form update.




\section{Discussion}
\label{sec:Discussions}

Interestingly, the quantum extension of Kelly's portfolio~\eqref{def:Kelly} corresponds exactly to maximum-likelihood quantum state tomography, estimating an unknown quantum state given measurement outcomes by the maximum-likelihood principle \citep{Hradil2004,Scholten2018}.
It is hence natural to consider extending our framework to the quantum setting.
The associated optimization problem is directly obtained by replacing all vectors in Kelly's portfolio problem with Hermitian matrices.
However, entry-wise multiplication and division between vectors appear in the MGM~\eqref{alg:IntroMGM},  the second-order characterization~\eqref{ineq:IntroHessianBounds}, and the construction of the mirror descent iterates in Lemma~\ref{lemma:MGMEqConjugateMD}.
There are multiple ways to generalize these operations to the quantum setting, where vectors are replaced by Hermitian matrices.
Whether any such generalization can be used to extend our framework to the quantum setting remains an open problem.

\section*{AI Disclosure}
The authors' main contributions lie in identifying the connections among the optimization problems listed in Section~\ref{sec:Intro}, reformulating existing algorithms as the MGM, and identifying the second-order characterization~\eqref{ineq:IntroHessianBounds}.
ChatGPT constructed the associated surrogate function and established its relative smoothness and relative strong convexity properties.

    






\ifbool{versionArXiv}{
\section*{Acknowledgments}
The authors would like to thank Wei-Fu Tseng for useful discussions. This work was supported by the National Science and Technology Council of Taiwan under Grant Nos. NSTC 112-2628-E-002-019-MY3 and NSTC 115-2221-E-002-208, by the research project “Geometry of Quantum Learning and Optimization” at National Taiwan University under Grant No. NTU-CC-115L893706, and by the Young Theoretical Scientist Award of the Mathematics Division of the National Center for Theoretical Sciences of Taiwan.
}{
\acks{
The authors would like to thank Wei-Fu Tseng for useful discussions. This work was supported by the National Science and Technology Council of Taiwan under Grant Nos. NSTC 112-2628-E-002-019-MY3 and NSTC 115-2221-E-002-208, by the research project “Geometry of Quantum Learning and Optimization” at National Taiwan University under Grant No. NTU-CC-115L893706, and by the Young Theoretical Scientist Award of the Mathematics Division of the National Center for Theoretical Sciences of Taiwan.
}
}

\ifbool{versionArXiv}{
    \bibliography{Contents/list}
}{
    \bibliography{../ArXiv-Version/Contents/list}
}

\appendix

\section{Counterexample for Log-Gradient Convexity}
\label{appendix:LogGradConvexCounterexample}
As discussed in Section~\ref{sec:Intro}, it is known that if the mapping $x\mapsto \log(-\nabla f(x)[j])$ is convex for every coordinate $j$, then the MGM converges at a rate of $\widetilde{O}(1/T)$ for minimizing $f$.
This section provides a counterexample to this log-gradient convexity for the Lewis weight problem~\eqref{def:Lewis}.

Consider the instance with $p=1$, $m=2$, $n=1$, and
$A=(1,1)^{\top}$.
We have
\begin{align}
    f_{\text{Lew}}(x)
    =\log\left(\frac{1}{x[1]}+\frac{1}{x[2]}\right),
    \quad\text{for all } x\in\mathbb{R}_{++}^2.
\end{align}

Furthermore, let $g_1(x)\coloneqq \log(-\nabla f_{\text{Lew}}(x)[1])=\log(x[2])-\log(x[1])-\log(x[1]+x[2])$.
Then, we have $\nabla^2 g_1(x)[2][2] = -x[2]^{-2}+ (x[1]+x[2])^{-2}< 0$ for all $x\in\mathbb{R}_{++}^2$.
This shows that $g_1(x)$ is not convex, thus providing a counterexample to log-gradient convexity for the Lewis weight problem.

\section{Proofs of Main Results}
In 
Appendix~\ref{appendix:ProofConjugateFunctionWellDefine},
Appendix~\ref{appendix:ProofRelativeSmoothnessStrongConvexity},
Appendix~\ref{appendix:ProofMGMEqConjugateMD}, 
Appendix~\ref{appendix:ProofConjugateMDConv},
Appendix~\ref{appendix:ProofEquivalenceOpt},
and Appendix~\ref{appendix:ProofConvergenceRateMGM}, 
we prove 
Lemma~\ref{lemma:ConjugateFunctionWellDefine}, 
Lemma~\ref{lemma:RelativeSmoothnessStrongConvexity}, Lemma~\ref{lemma:MGMEqConjugateMD}, 
Corollary~\ref{corollary:ConjugateMDConv},
Lemma~\ref{lemma:EquivalenceOpt},
and Theorem~\ref{theorem:ConvergenceRateMGM}, respectively, which are the main results of our framework stated in Section~\ref{sec:MainResults}.

\subsection{Proof of Lemma~\ref{lemma:ConjugateFunctionWellDefine}}
\label{appendix:ProofConjugateFunctionWellDefine}

The proof of Lemma~\ref{lemma:ConjugateFunctionWellDefine} is based on 
Lemma~\ref{lemma:ZeroOrderCharacterization}
and Lemma~\ref{lemma:USupLinear}.

\begin{lemma}
    \label{lemma:ZeroOrderCharacterization}
    Let $\theta\in\Set{-1,1}$ and $0\leq\mu\leq L$, and let $f$ be any $\theta$-log-homogeneous, twice differentiable function.
    Furthermore, define the function $U_{\theta,c}^{f}$ as follows:
    \begin{align}
        U_{\theta,c}^{f}(y)
        \coloneqq
        \begin{cases}
            \frac{\exp\left((c+\theta)
            f(y^{\theta/(c+\theta)})\right)}{c+\theta},
            &\quad\text{if }c\neq-\theta,\\
            f(\exp(y)),
            &\quad\text{if }c=-\theta=1,
        \end{cases}
    \end{align}
    where the domain of $U_{\theta,c}^{f}$ is $\mathbb{R}_{++}^m$ in the first case and $\mathbb{R}^m$ in the second case.
    Then, $f$ satisfies the  second-order characterization~\eqref{ineq:IntroHessianBounds} if and only if $U_{\theta,L}^{f}$ is concave and $U_{\theta,\mu}^{f}$ is convex.
\end{lemma}
\begin{proof}
    Since $U_{\theta,c}^f$ is twice differentiable, its concavity is equivalent to
    $\nabla^2 U_{\theta,c}^f(y)\leq0$ throughout its domain,
    and its convexity is equivalent to
    $\nabla^2 U_{\theta,c}^f(y)\geq0$ throughout its domain.

    For the case $c\neq -\theta$, we compute the Hessian matrix $\nabla^2 U_{\theta,c}^f(y)$ as follows:
    \begin{align}
        \label{eq:HessianUthetaNeqNegTheta}
        \nabla^2 U_{\theta,c}^f(y)
        =\frac{U_{\theta,c}^f(y)}{c+\theta}\mathrm{Diag}\left(y^{-c/(\theta+c)}\right)\left(H_1-cH_2\right)\mathrm{Diag}\left(y^{-c/(\theta+c)}\right),
    \end{align}
    where 
    \begin{align}
        H_1
        =\nabla^2 f(y^{\theta/(\theta+ c)})+\theta\nabla f(y^{\theta/(\theta+ c)})\nabla f(y^{\theta/(\theta+ c)})^{\top},
    \end{align}
    and
    \begin{align}
        H_2
        =\mathrm{Diag}\left(\theta \nabla f(y^{\theta/(\theta+ c)})\oslash y^{\theta/(\theta+c)}\right)-\nabla f(y^{\theta/(\theta+ c)})\nabla f(y^{\theta/(\theta+ c)})^{\top}.
    \end{align}

    For the case $c=-\theta=1$, we have
    \begin{align}
        \label{eq:HessianUthetaEqNegTheta}
        \nabla^2 U_{-1,1}^f(y)
        =\mathrm{Diag}\left(\exp(y)\right)\left(\bar{H}_1-\bar{H}_2\right)\mathrm{Diag}\left(\exp(y)\right),
    \end{align}
    where
    \begin{align}
        \bar{H}_1
        =\nabla^2 f(\exp( y)),
    \end{align}
    and
    \begin{align}
        \bar{H}_2
        =\mathrm{Diag}\left(-\nabla f(\exp( y))\oslash \exp(y)\right).
    \end{align}

    Since $\mathrm{Diag}(y^{-c/(\theta+c)})$ and $\mathrm{Diag}(\exp(y))$ are both invertible, the concavity or convexity of $U_{\theta,c}^f$ is determined by the definiteness of $H_1-cH_2$ and $\bar{H}_1-\bar{H}_2$, respectively.

    The proof then concludes by substituting $x=y^{\theta/(\theta+c)}$  and $x=\exp( y)$ in the equation~\eqref{eq:HessianUthetaNeqNegTheta} and in the equation~\eqref{eq:HessianUthetaEqNegTheta}, respectively.
\end{proof}

\begin{lemma}
    \label{lemma:USupLinear}
    Let $f$ be as in Lemma~\ref{lemma:RelativeSmoothnessStrongConvexity}.
    If $\theta+L\neq 0$, then we have
    \begin{align}
        U_{\theta,L}^{f}(y)
        =\inf_{z\in\mathbb{R}_{++}^m}\inner{\nabla U_{\theta,L}^{f}(z)}{y},\quad\text{for all } y\in\mathbb{R}_{++}^m,
    \end{align}
    where $U_{\theta,c}^f$ is defined as in Lemma~\ref{lemma:ZeroOrderCharacterization}.
    On the other hand, if $\theta+\mu\neq 0$, then we have
    \begin{align}
        U_{\theta,\mu}^{f}(y)
        =\sup_{z\in\mathbb{R}_{++}^m}\inner{\nabla U_{\theta,\mu}^{f}(z)}{y},\quad\text{for all } y\in\mathbb{R}_{++}^m.
    \end{align}
\end{lemma}
The proof of Lemma~\ref{lemma:USupLinear} is based on Lemma~\ref{lemma:SelfNormalize}.
\begin{lemma}
    \label{lemma:SelfNormalize}
    Let $f$ be differentiable and $\theta$-log-homogeneous on $\mathbb{R}_{++}^m$.
    Then, we have $\inner{\nabla f(x)}{x}=\theta$ for all $x\in\mathbb{R}_{++}^m$.
\end{lemma}
\begin{proof}
    From the log-homogeneity of $f$, we have $f(rx)=f(x)+\theta\log(r)$ for all $r>0$ and $x\in\mathbb{R}_{++}^m$.
    Differentiating both sides with respect to $r$, we obtain
    \begin{align}
        \inner{\nabla f(rx)}{x}
        =\frac{\theta}{r}.
    \end{align}
    The proof then concludes by setting $r=1$.
\end{proof}

\begin{proofof}{Lemma~\ref{lemma:USupLinear}}
    By Lemma~\ref{lemma:ZeroOrderCharacterization}, we know that $U_{\theta,L}^{f}$ is concave and $U_{\theta,\mu}^{f}$ is convex.

    By the concavity of $U_{\theta,L}^{f}$, for all $y\in\mathbb{R}_{++}^m$, we have
    \begin{align}
        U_{\theta,L}^{f}(y)
        &\leq U_{\theta,L}^{f}(z)+\inner{\nabla U_{\theta,L}^{f}(z)}{y-z}\nonumber\\
        &=\frac{\exp\left((L+\theta)f\left(z^{\theta/(L+\theta)}\right)\right)}{L+\theta}\nonumber\\
        &\quad+\inner{\frac{\theta\exp\left((L+\theta)f\left(z^{\theta/(L+\theta)}\right)\right)\nabla f\left(z^{\theta/(L+\theta)}\right)\odot z^{-L/(L+\theta)}}{L+\theta}}{y-z}\nonumber\\
        &=\inner{\frac{\theta\exp\left((L+\theta)f\left(z^{\theta/(L+\theta)}\right)\right)\nabla f\left(z^{\theta/(L+\theta)}\right)\odot z^{-L/(L+\theta)}}{L+\theta}}{y}\nonumber\\
        &=\inner{\nabla U_{\theta,L}^{f}(z)}{y},
    \end{align}
    where the second equality follows from Lemma~\ref{lemma:SelfNormalize} and the inequality becomes equality when $y=z$.
    Taking the infimum over all $z\in\mathbb{R}_{++}^m$ on the right-hand side completes the proof of the first part of the lemma.

    Similarly, by the convexity of $U_{\theta,\mu}^{f}$, for all $y\in\mathbb{R}_{++}^m$, we have
    \begin{align}
        U_{\theta,\mu}^{f}(y)
        &\geq U_{\theta,\mu}^{f}(z)+\inner{\nabla U_{\theta,\mu}^{f}(z)}{y-z}\nonumber\\
        &=\inner{\nabla U_{\theta,\mu}^{f}(z)}{y},\nonumber
    \end{align}
    where the equality follows from Lemma~\ref{lemma:SelfNormalize} and the inequality becomes equality when $y=z$.
    Taking the supremum over all $z\in\mathbb{R}_{++}^m$ on the right-hand side completes the proof.
\end{proofof}

\begin{proofof}{Lemma~\ref{lemma:ConjugateFunctionWellDefine}}
    By Lemma~\ref{lemma:ZeroOrderCharacterization}, we know that $U_{\theta,L}^{f}$ is concave, and $U_{\theta,\mu}^{f}$ is convex.

    For the case $\theta=-1$ and $\mu\leq L=1$, as well as the case $\theta=-1$ and $1=\mu< L$, the convexity of $\varphi$ directly follows from the concavity of $U_{\theta,L}^{f}$ and the convexity of $U_{\theta,\mu}^{f}$, respectively.

    For the case $\theta=-1$ and $L<1$, we have
    \begin{align}
        \varphi(y)
        &=-f(\exp(y))\nonumber\\
        &=\frac{\log\left((L-1)U_{-1,L}^{f}\left(\exp((1-L)y)\right)\right)}{1-L}\nonumber
        \\
        &=\frac{\log\left((L-1)\left(\inf_{z\in\mathbb{R}_{++}^m}\inner{\nabla U_{-1,L}^{f}(z)}{\exp((1-L)y)}\right)\right)}{1-L}\nonumber
        \\
        &=\sup_{z\in\mathbb{R}_{++}^m}\frac{\log\left((L-1)\left(\inner{\nabla U_{-1,L}^{f}(z)}{\exp((1-L)y)}\right)\right)}{1-L},\nonumber
    \end{align}
    where the third equality follows from Lemma~\ref{lemma:USupLinear}.
    We note that, for any fixed $z\in\mathbb{R}_{++}^m$, the function $\log\left((L-1)\left(\inner{\nabla U_{-1,L}^{f}(z)}{\exp((1-L)y)}\right)\right)$ is convex in $y$.
    The convexity of $\varphi(y)$ then follows from the fact that the supremum of convex functions is convex.

    Similarly, for the case $\theta=-1$ and $\mu>1$, as well as the case $\theta=1$, we have
    \begin{align}
        \varphi(y)
        &=f(\exp(\theta y))\nonumber\\
        &=\frac{\log\left((\mu+\theta)U_{\theta,\mu}^{f}\left(\exp((\mu+\theta)y)\right)\right)}{\mu+\theta}\nonumber
        \\
        &=\frac{\log\left((\mu+\theta)\left(\sup_{z\in\mathbb{R}_{++}^m}\inner{\nabla U_{\theta,\mu}^{f}(z)}{\exp((\mu+\theta)y)}\right)\right)}{\mu+\theta}\nonumber
        \\
        &=\sup_{z\in\mathbb{R}_{++}^m}\frac{\log\left((\mu+\theta)\left(\inner{\nabla U_{\theta,\mu}^{f}(z)}{\exp((\mu+\theta)y)}\right)\right)}{\mu+\theta},\nonumber
    \end{align}
    which is also a supremum of convex functions, and hence convex in $y$.
    This concludes the proof.
\end{proofof}

\subsection{Proof of Lemma~\ref{lemma:RelativeSmoothnessStrongConvexity}}
\label{appendix:ProofRelativeSmoothnessStrongConvexity}

\begin{proof}
    For $\theta=-1$ and $\mu=L=1$,
    by Lemma~\ref{lemma:ZeroOrderCharacterization},
    $U_{-1,1}^f(y)=f(\exp(y))$ is both convex and concave,
    and is therefore affine. 
    Hence, we have $\varphi(y)=\inner{a}{y}+b$ for some constant vector $a$ and scalar $b$.
    By the assumption that $\theta\nabla f(x)>0$ for all $x\in\mathbb{R}_{++}^m$, as well as Lemma~\ref{lemma:SelfNormalize}, we must have $a=\nabla \varphi(y)=-\exp(y)\odot\nabla f(\exp(y))\in\mathrm{ri}(\Delta_m)$.
    This implies that $\mathcal{C}=\{a\}$ is a singleton.
    Hence, Lemma~\ref{lemma:RelativeSmoothnessStrongConvexity} is trivially satisfied for this special case.

    The relative smoothness is trivially satisfied for the special case $\theta=-1$ and $\mu<L=1$.
    This is because $L\varphi^*(q)-G^f(q)=h(q)$ is convex, where $h(q)=\inner{q}{\log(q)}$.

    Similarly, the relative strong convexity is also trivially satisfied for the special case $\theta=-1$ and $1=\mu<L$, since $G^f(q)-\mu\varphi^*(q)=h(q)$.

    To establish the relative smoothness, it remains to consider the case $\theta=-1$ and $L\neq 1$, and the case $\theta=1$.
    Similarly, to establish the relative strong convexity, it remains to consider the case $\theta=-1$ and $\mu\neq 1$, and the case $\theta=1$.
    
    To prove the relative smoothness for the case $\theta=-1$ and $L< 1$ is equivalent to proving the convexity of the function $h(q)-(1-L)\varphi^{*}(q)$, for which we will use the convexity of the function $(L-1)U_{-1,L}^f(y)$ established in Lemma~\ref{lemma:ZeroOrderCharacterization}.

    In particular, we write
    \begin{align}
        \label{eq:ScalarInfPrototype}
        &h(q)-(1-L)\varphi^{*}(q)\nonumber\\
        &\quad=\inf_{y\in\mathbb{R}^{m}}h(q)-(1-L)\inner{q}{y}-(1-L)f(\exp(y))\nonumber\\
        &\quad=\inf_{w\in\mathbb{R}_{++}^{m}}\inner{q}{\log(q)-\log(w)}-(1-L)f(w^{1/(1-L)})\nonumber\\
        &\quad=\inf_{w\in\mathbb{R}_{++}^{m}}\inner{q}{\log(q)-\log(w)}+\log((L-1)U_{-1,L}^f(w))\nonumber\\
        &\quad=\inf_{w\in\mathbb{R}_{++}^{m},t>0}\inner{q}{\log(q)-\log(w)}+t(L-1)U_{-1,L}^f(w)-\log(t)-1\nonumber\\
        &\quad=\inf_{w\in\mathbb{R}_{++}^{m},t>0}\inner{q}{\log(q)-\log(tw)}+(L-1)U_{-1,L}^f(tw)-1\nonumber\\
        &\quad=\inf_{z\in\mathbb{R}_{++}^{m}}\inner{q}{\log(q)-\log(z)}+(L-1)U_{-1,L}^f(z)-1,\quad\text{for all } q\in\Delta_m,
    \end{align}
    where the fourth equality follows from the fact that $\log(u)=\inf_{t>0}tu-\log(t)-1$ for all $u>0$, and the fifth equality follows from the fact that $U_{-1,L}^f(tw)=tU_{-1,L}^f(w)$ due to the log-homogeneity of $f$.
    The convexity of $h(q)-(1-L)\varphi^{*}(q)$ then follows from the joint convexity of the function $J(q,z)=\inner{q}{\log(q)-\log(z)}+(L-1)U_{-1,L}^f(z)-1$.
    This concludes the proof of the relative smoothness for the case $\theta=-1$ and $L< 1$.

    Similarly, to prove the relative strong convexity for the case $\theta=-1$ and $\mu> 1$, as well as the case $\theta=1$, we write
    \begin{align}
        &h(q)-(\mu+\theta)\varphi^{*}(q)\nonumber\\
        &\quad=\inf_{y\in\mathbb{R}^{m}}h(q)-(\mu+\theta)\inner{q}{y}+(\mu+\theta)f(\exp(\theta y))\nonumber\\
        &\quad=\inf_{w\in\mathbb{R}_{++}^{m}}\inner{q}{\log(q)-\log(w)}+(\mu+\theta)f(w^{\theta/(\theta+\mu)})\nonumber\\
        &\quad=\inf_{w\in\mathbb{R}_{++}^{m}}\inner{q}{\log(q)-\log(w)}+\log((\theta+\mu)U_{\theta,\mu}^f(w))\nonumber\\
        &\quad=\inf_{z\in\mathbb{R}_{++}^{m}}\tilde{J}(q,z),
    \end{align}
    where $\tilde{J}(q,z)=\inner{q}{\log(q)-\log(z)}+(\theta+\mu)U_{\theta,\mu}^f(z)-1$ and the last equality follows from the same argument as in the equality~\eqref{eq:ScalarInfPrototype}.
    Then, the desired relative strong convexity follows from the joint convexity of $\tilde{J}(q,z)$.

    On the other hand, to prove the relative strong convexity for the case $\theta=-1$ and $\mu<1$ is equivalent to proving the convexity of the function $(1-\mu)\varphi^{*}(q)-h(q)$.
    By Lemma~\ref{lemma:USupLinear}, we write
    \begin{align}
        \label{eq:LinearMaxPrototype}
        (1-\mu)\varphi^{*}(q)-h(q)
        &=\sup_{y\in\mathbb{R}^{m}}(1-\mu)\inner{q}{y}+(1-\mu)f(\exp(y))-h(q)\nonumber\\
        &=\sup_{w\in\mathbb{R}_{++}^{m}}\inner{q}{\log(w)}+(1-\mu)f(w^{1/(1-\mu)})-h(q)\nonumber\\
        &=\sup_{w\in\mathbb{R}_{++}^{m}}\inner{q}{\log(w)-\log(q)}-\log((\mu-1)U_{-1,\mu}^f(w))\nonumber\\
        &=\sup_{w\in\mathbb{R}_{++}^{m}}\inner{q}{\log(w)-\log(q)}-\inf_{s\in\mathbb{R}_{++}^{m}}\log((\mu-1)\inner{\nabla U_{-1,\mu}^f(s)}{w})\nonumber\\
        &=\sup_{s\in\mathbb{R}_{++}^{m}}\sup_{w\in\mathbb{R}_{++}^{m}}\inner{q}{\log(w)-\log(q)}-\log((\mu-1)\inner{\nabla U_{-1,\mu}^f(s)}{w})\nonumber\\
        &=\sup_{s\in\mathbb{R}_{++}^{m}}\inner{q}{-\log((\mu-1)\nabla U_{-1,\mu}^f(s))},\quad\text{for all } q\in\Delta_m,
    \end{align}
    which is a supremum of linear functions.
    Hence, $(1-\mu)\varphi^{*}(q)-h(q)$ is convex.

    Similarly, to prove the relative smoothness for the case $\theta=-1$ and $L> 1$, as well as the case $\theta=1$, we write
    \begin{align}
        (L+\theta)\varphi^{*}(q)-h(q)
        &=\sup_{y\in\mathbb{R}^{m}}(L+\theta)\inner{q}{y}-(L+\theta)f(\exp(\theta y))-h(q)\nonumber\\
        &=\sup_{w\in\mathbb{R}_{++}^{m}}\inner{q}{\log(w)}-(L+\theta)f(w^{\theta/(L+\theta)})-h(q)\nonumber\\
        &=\sup_{w\in\mathbb{R}_{++}^{m}}\inner{q}{\log(w)-\log(q)}-\log((L+\theta)U_{\theta,L}^f(w))\nonumber\\
        &=\sup_{s\in\mathbb{R}_{++}^m}
        \inner{q}{-\log((L+\theta)\nabla U_{\theta,L}^f(s))},\nonumber
    \end{align}
    where the last equality follows from the same argument as in the equality~\eqref{eq:LinearMaxPrototype}.
    The desired relative smoothness follows from the fact that $(L+\theta)\varphi^{*}(q)-h(q)$ is a supremum of linear functions, and hence convex.
\end{proof}

\subsection{Proof of Lemma~\ref{lemma:MGMEqConjugateMD}}
\label{appendix:ProofMGMEqConjugateMD}
The proof of Lemma~\ref{lemma:MGMEqConjugateMD} is based on
Lemma~\ref{lemma:FenchelSubgradient} and
Lemma~\ref{lemma:RelativeSmoothnessStrongConvexitySubgradient}.

\begin{lemma}[{[\citealp{Rockafellar1970b}, Theorem~23.5 and Corollary~23.5.1]}]
    \label{lemma:FenchelSubgradient}
    For any function $\tilde{\varphi}$ that is closed, proper, and convex, we have
    $x\in\partial \tilde{\varphi}^{*}(x^{*})$ if and only if $x^{*}\in\partial \tilde{\varphi}(x)$.
    Furthermore, if $x^{*}\in\partial \tilde{\varphi}(x)$, then $\tilde{\varphi}(x)+\tilde{\varphi}^{*}(x^{*})=\inner{x}{x^{*}}$.
\end{lemma}

\begin{lemma}
    \label{lemma:RelativeSmoothnessStrongConvexitySubgradient}
    Let $r_t$ and $g_t$ be defined as in Lemma~\ref{lemma:MGMEqConjugateMD}.
    Then, we have $g_t-\mu r_t \in \partial (G^f-\mu \varphi^{*})(q_t)$ and $Lr_t-g_t\in \partial (L\varphi^{*}-G^f)(q_t)$.
    Here, both differences of functions are defined on $\mathcal{C}$ and extended to take the value $\infty$ outside $\mathcal{C}$.
\end{lemma}
\begin{proof}
    We first prove that $q_t\in\mathcal{C}$.
    In particular, the fact that $q_t=\theta x_t\odot\nabla f(x_t)\in\mathrm{ri}(\Delta_m)$ follows from Lemma~\ref{lemma:SelfNormalize} and the assumption that $\theta\nabla f(x)>0$ for all $x\in\mathbb{R}_{++}^m$.
    On the other hand, by Lemma~\ref{lemma:FenchelSubgradient}, we have 
    \begin{align}
        \varphi^*(q_t)
        =\inner{q_t}{c_f\theta\log(x_t)}
        -\varphi(c_f\theta\log(x_t))
        <\infty.
    \end{align}
    This shows that $q_t\in\mathcal{C}$.

    For $\theta=-1$ and $\mu=L=1$,
    as established in the proof of Lemma~\ref{lemma:RelativeSmoothnessStrongConvexity} in Appendix~\ref{appendix:ProofRelativeSmoothnessStrongConvexity}, $\mathcal{C}=\{a\}$ is a singleton.
    Hence, Lemma~\ref{lemma:RelativeSmoothnessStrongConvexitySubgradient} is trivially satisfied for this special case.

    The property  $Lr_t - g_t \in \partial (L\varphi^{*} - G^{f})(q_t)$ is trivially satisfied for the special case $\theta=-1$ and $\mu<L=1$, since $L\varphi^*(q)-G^f(q)=h(q)$ and $Lr_t-g_t=\log(q_t)+\mathbf{1}_m=\nabla h(q_t)\in\partial h(q_t)$, where $h(q)=\inner{q}{\log(q)}$.
    Similarly, the property $ g_t-\mu r_t  \in \partial (G^{f}-\mu\varphi^{*} )(q_t)$ is also trivially satisfied for the special case $\theta=-1$ and $1=\mu<L$, since $G^f(q)-\mu\varphi^*(q)=h(q)$.

    To establish the property  $Lr_t - g_t \in \partial (L\varphi^{*} - G^{f})(q_t)$, it remains to consider the case $\theta=-1$ and $L\neq 1$, and the case $\theta=1$.
    Similarly, to establish  the property $ g_t-\mu r_t  \in \partial (G^{f}-\mu\varphi^{*} )(q_t)$, it remains to consider the case $\theta=-1$ and $\mu\neq 1$, and the case $\theta=1$.
    
    To prove that $Lr_t - g_t \in \partial (L\varphi^{*} - G^{f})(q_t)$ for the case $\theta=-1$ and $L< 1$ is equivalent to verifying that $Lr_t-g_t\in\partial(h-(1-L)\varphi^{*})(q_t)$.
    In particular, let $J(\cdot,\cdot)$ be defined as in Appendix~\ref{appendix:ProofRelativeSmoothnessStrongConvexity}.
    Furthermore, let $\nabla_q J(\cdot,\cdot)$ and $\nabla_z J(\cdot,\cdot)$ denote the gradients of $J$ with respect to its first and second arguments, respectively.
    Assume that $\nabla_z J(q_t,z_t)=0$ for  $z_t=\exp\left((1-L)f(x_t)\right)x_t^{1-L}$.
    Then, we have $h(q_t)-(1-L)\varphi^{*}(q_t)=J(q_t,z_t)$.
    From the convexity of $J$, we have
    \begin{align}
        &J(q,z)\nonumber\\
        &\quad\geq J(q_t,z_t) + \inner{\nabla_q J(q_t,z_t)}{q-q_t} + \inner{\nabla_z J(q_t,z_t)}{z-z_t}\nonumber\\
        &\quad= h(q_t)-(1-L)\varphi^{*}(q_t) + \inner{\nabla_q J(q_t,z_t)}{q-q_t}\nonumber\\
        &\quad= h(q_t)-(1-L)\varphi^{*}(q_t) + \inner{\log(q_t)-\log(z_t)}{q-q_t}\nonumber\\
        &\quad= h(q_t)-(1-L)\varphi^{*}(q_t) + \inner{Lr_t-g_t}{q-q_t},\quad \text{for all } q\in\Delta_m,
    \end{align}
    where the last equality follows from the fact that $\inner{\mathbf{1}_m}{q-q_t}=0$ for all $q\in\Delta_m$.
    By taking the infimum over $z$ on both sides of the inequality, we obtain
    \begin{align}
        \label{eq:ScalarInfPrototypeSubgradient}
        h(q)-(1-L)\varphi^{*}(q)
        &\geq h(q_t)-(1-L)\varphi^{*}(q_t) + \inner{Lr_t-g_t}{q-q_t},\quad \text{for all } q\in\Delta_m.
    \end{align}
    This implies that $Lr_t-g_t$ is a subgradient of $h-(1-L)\varphi^*$ at $q_t$, provided that  $\nabla_z J(q_t,z_t)=0$ for  $z_t=\exp\left((1-L)f(x_t)\right)x_t^{1-L}$.

    To prove that $\nabla_z J(q_t,z_t)=0$, we write
    \begin{align}
        (L-1)U_{-1,L}^f(z_t)
        &=\exp\left((L-1)f(\exp(f(x_t))x_t)\right)\nonumber\\
        &=\exp\left((L-1)(f(x_t)-\log(\exp(f(x_t))))\right)\nonumber\\
        &=1,
    \end{align}
    where the second equality follows from the log-homogeneity of $f$.
    Then, we have
    \begin{align}
        \label{eq:ScalarInfPrototypeGrad0}
        \nabla_z J(q_t,z_t) 
        &= -q_t\oslash z_t + (L-1) \nabla U_{-1,L}^f(z_t) \nonumber\\
        &=\exp\left((L-1)f(x_t)\right)\nabla f(x_t)\odot x_t^{L} -(L-1)U_{-1,L}^f(z_t)\nabla f(z_t^{1/(1-L)})\odot z_t^{L/(1-L)}\nonumber\\
        &=\exp\left((L-1)f(x_t)\right)\nabla f(x_t)\odot x_t^{L} -\nabla f(z_t^{1/(1-L)})\odot z_t^{L/(1-L)}\nonumber\\
        &=0,
    \end{align}
    where the last equality follows from the equality $\nabla f(cx)=\nabla f(x)/c$ due to the log-homogeneity of $f$.
    This concludes the proof of the property $Lr_t - g_t \in \partial (L\varphi^{*} - G^{f})(q_t)$ for the case $\theta=-1$ and $L< 1$.

    Similarly, to prove that $g_t-\mu r_t  \in \partial (G^{f}-\mu\varphi^{*} )(q_t)$ for the case $\theta=-1$ and $\mu> 1$, as well as the case $\theta=1$, let $\tilde{J}(\cdot,\cdot)$ be defined as in Appendix~\ref{appendix:ProofRelativeSmoothnessStrongConvexity}.
    Then, by the same arguments as in the equality~\eqref{eq:ScalarInfPrototypeGrad0} and the inequality~\eqref{eq:ScalarInfPrototypeSubgradient}, we have $\nabla_z\tilde{J}(q_t,\tilde{z}_t)=0$ for $\tilde{z}_t=\exp\left(-(\mu+\theta)f(x_t)\right)x_t^{\theta(\mu+\theta)}$, and $g_t-\mu r_t$ is also a subgradient of $h-(\mu+\theta)\varphi^*$ at $q_t$.
    This proves the property $g_t-\mu r_t  \in \partial (G^{f}-\mu\varphi^{*} )(q_t)$ for the case $\theta=-1$ and $\mu>1$, as well as the case $\theta=1$.

    On the other hand, to prove that $g_t-\mu r_t  \in \partial (G^{f}-\mu\varphi^{*} )(q_t)$ for the case $\theta=-1$ and $\mu<1$ is equivalent to verifying that $g_t-\mu r_t\in\partial((1-\mu)\varphi^{*}-h)(q_t)$.
    In particular, for $s_t=x_t^{1-\mu}$, we have
    \begin{align}
        &\inner{q_t}{-\log((\mu-1)\nabla U_{-1,\mu}^f(s_t))}\nonumber\\
        &\quad=\inner{q_t}{(1-\mu)\log(x_t)-\log(q_t)-\log((\mu-1)U_{-1,\mu}^f(s_t))\mathbf{1}_m}\nonumber\\
        &\quad=(1-\mu)\inner{q_t}{\log(x_t)}+(1-\mu)f(x_t)-h(q_t)\nonumber\\
        &\quad=(1-\mu)(\inner{q_t}{\log(x_t)}-\varphi(\log(x_t)))-h(q_t)\nonumber\\
        &\quad=(1-\mu)\varphi^{*}(q_t)-h(q_t),
    \end{align}
    where the last equality follows from Lemma~\ref{lemma:FenchelSubgradient}, Lemma~\ref{lemma:ConjugateFunctionWellDefine}, and the fact that $q_t=\nabla\varphi(\log(x_t))$.
    Combining this with the equality~\eqref{eq:LinearMaxPrototype} in Appendix~\ref{appendix:ProofRelativeSmoothnessStrongConvexity}, we obtain that
    \begin{align}
        \label{eq:LinearMaxPrototypeSubgradient}
        (1-\mu)\varphi^{*}(q)-h(q)
        &=\sup_{s\in\mathbb{R}_{++}^{m}}\inner{q}{-\log((\mu-1)\nabla U_{-1,\mu}^f(s))}\nonumber\\
        &\geq \inner{q}{-\log((\mu-1)\nabla U_{-1,\mu}^f(s_t))}\nonumber\\
        &=(1-\mu)\varphi^{*}(q_t)-h(q_t)+\inner{q-q_t}{-\log((\mu-1)\nabla U_{-1,\mu}^f(s_t))}\nonumber\\
        &=(1-\mu)\varphi^{*}(q_t)-h(q_t)+\inner{q-q_t}{g_t-\mu r_t},\quad\text{for all } q\in\Delta_m.
    \end{align}
    This implies that $g_t-\mu r_t$ is a subgradient of $(1-\mu)\varphi^*-h$ at $q_t$.

    Similarly, to prove that $Lr_t - g_t \in \partial (L\varphi^{*} - G^{f})(q_t)$ for the case $\theta=-1$ and $L> 1$, as well as the case $\theta=1$, 
    we write
    \begin{align}
        \inner{q_t}{-\log((L+\theta)\nabla U_{\theta,L}^f(\tilde{s}_t))}
        &=\inner{q_t}{\theta(\theta+L)\log(x_t)-\log(q_t)-\log((L+\theta)U_{\theta,L}^f(\tilde{s}_t))\mathbf{1}_m}\nonumber\\
        &=(\theta+L)\inner{q_t}{\theta\log(x_t)}-(\theta+L)f(x_t)-h(q_t)\nonumber\\
        &=(\theta+L)(\inner{q_t}{\theta\log(x_t)}-\varphi(\theta\log(x_t)))-h(q_t)\nonumber\\
        &=(\theta+L)\varphi^{*}(q_t)-h(q_t),
    \end{align}
    where $\tilde{s}_t=x_t^{\theta(L+\theta)}$, and the last equality follows from Lemma~\ref{lemma:FenchelSubgradient}, Lemma~\ref{lemma:ConjugateFunctionWellDefine}, and the fact that $q_t=\nabla\varphi(\theta\log(x_t))$.
    The proof then follows by the same argument as in the inequality~\eqref{eq:LinearMaxPrototypeSubgradient}.
\end{proof}

\begin{proofof}{Lemma~\ref{lemma:MGMEqConjugateMD}}
    The fact that $q_t\in\mathcal{C}$ is already established in the proof of Lemma~\ref{lemma:RelativeSmoothnessStrongConvexitySubgradient}.

    From Lemma~\ref{lemma:FenchelSubgradient}, we have $r_t\in\partial \varphi^{*}(q_t)$.
    
    On the other hand, by Lemma~\ref{lemma:RelativeSmoothnessStrongConvexitySubgradient}, we have
    \begin{align}
        G^f(q)-G^f(q_t)
        &\geq \inner{g_t}{q-q_t} + \mu  D_{\varphi^*}(q,q_t)\nonumber\\
        &\geq  \inner{g_t}{q-q_t},\quad\text{for all } q\in\mathcal{C}.
    \end{align}
    This implies that $g_t\in\partial G^f(q_t)$.

    It remains to show that, for all $t\geq 1$,
    \begin{align}
        q_{t+1}
        \in \argmin_{q\in\Delta_m}\eta\inner{\nabla G^f(q_t)}{q}+D_{\varphi^{*}}(q,q_t),
    \end{align}
    where $\nabla G^f(q_t)=g_t$ and $\nabla \varphi^{*}(q_t)=r_t$.
    This is equivalent to showing that $q_{t+1}$ is a minimizer of the function $\varPsi_t(q)=\inner{\eta g_t-r_t}{q}+\varphi^{*}(q)$ over $\Delta_m$.

    Since $\varPsi_t(q)-\varphi^{*}(q)$ is a linear function of $q$, we have $r_{t+1}+\eta g_t-r_t\in\partial \varPsi_t(q_{t+1})$.
    Furthermore, we have
    \begin{align}
        \label{eq:BregProxEqualityBackBone}
        r_{t+1}+\eta g_t-r_t
        &=c_f\theta\log(x_{t+1})-c_f\theta\log(x_t)\nonumber\\
        &\quad+\eta c_f(-\log(x_t)+\log(q_t)+\mathbf{1}_m)\nonumber\\
        &=c_f\theta\log\left(\frac{x_t\odot(\theta\nabla f(x_t))^{-\theta\eta}}{\norm{x_t\odot(\theta\nabla f(x_t))^{-\theta\eta}}_1}\right)-c_f\theta\log(x_t)\nonumber\\
        &\quad+\eta c_f(-\log(x_t)+\log(\theta x_t\odot\nabla f(x_t))+\mathbf{1}_m)\nonumber\\
        &=c_f\theta\left(-\log\left(\norm{x_t\odot(\theta\nabla f(x_t))^{-\theta\eta}}_1\right)+\eta\theta\right)\mathbf{1}_m.
    \end{align}

    Therefore, we have
    \begin{align}
        \varPsi_t(q)
        &\geq \varPsi_t(q_{t+1})+\inner{r_{t+1}+\eta g_t-r_t}{q-q_{t+1}}\nonumber\\
        &= \varPsi_t(q_{t+1})+\inner{c_f\theta\left(-\log\left(\norm{x_t\odot(\theta\nabla f(x_t))^{-\theta\eta}}_1\right)+\eta\theta\right)\mathbf{1}_m}{q-q_{t+1}}\nonumber\\
        &= \varPsi_t(q_{t+1}),\quad\text{for all } q\in\Delta_m.
    \end{align}
    This concludes the proof.
\end{proofof}

\subsection{Proof of Corollary~\ref{corollary:ConjugateMDConv}}
\label{appendix:ProofConjugateMDConv}
In the proof of Corollary~\ref{corollary:ConjugateMDConv}, instead of directly using the relative smoothness and strong convexity properties established in Lemma~\ref{lemma:RelativeSmoothnessStrongConvexity}, we will use the corresponding subgradient inequalities established in  Lemma~\ref{lemma:RelativeSmoothnessStrongConvexitySubgradient}.
\begin{proof}
    Let the subgradients $\nabla G^f(q_t)=g_t$ and $\nabla \varphi^{*}(q_t)=r_t$ be chosen as in Lemma~\ref{lemma:MGMEqConjugateMD}.

    We first establish below a version of the \emph{Bregman proximal inequality}
    from the analysis of \citet{Teboulle2018}, with gradients replaced by subgradients.
    In our setting, this inequality holds with equality:
    \begin{align}
        \label{eq:BregmanProximalInequalityMGMLemma}
        &D_{\varphi^{*}}(q,q_t)-D_{\varphi^{*}}(q,q_{t+1})-D_{\varphi^{*}}(q_{t+1},q_t)\nonumber\\
        &\quad=(\varphi^{*}(q)-\varphi^{*}(q_{t})-\inner{r_t}{q-q_{t}})\nonumber\\
        &\quad\quad-(\varphi^{*}(q)-\varphi^{*}(q_{t+1})-\inner{r_{t+1}}{q-q_{t+1}})\nonumber\\
        &\quad\quad-(\varphi^{*}(q_{t+1})-\varphi^{*}(q_{t})-\inner{r_{t}}{q_{t+1}-q_{t}})\nonumber\\
        &\quad=\inner{r_t-r_{t+1}}{q_{t+1}-q}\nonumber\\
        &\quad=\inner{c_f\theta\log(x_t)-c_f\theta\log(x_{t+1})}{q_{t+1}-q}\nonumber\\
        &\quad=\inner{c_f\theta\log(x_t)-c_f\theta\log\left(\frac{x_t\odot(\theta\nabla f(x_t))^{-\theta\eta}}{\norm{x_t\odot(\theta\nabla f(x_t))^{-\theta\eta}}_1}\right)}{q_{t+1}-q}\nonumber\\
        &\quad=\eta\inner{c_f\log(q_t)-c_f
        \log(x_t)+\frac{c_f\theta}{\eta}\log\left(\norm{x_t\odot(\theta\nabla f(x_t))^{-\theta\eta}}_1\right)\mathbf{1}_m}{q_{t+1}-q}\nonumber\\
        &\quad=\eta\inner{g_t}{q_{t+1}-q}\nonumber\\
        &\quad=\eta\inner{\nabla  G^{f}(q_t)}{q_{t+1}-q},\quad\text{for all } q\in\mathcal{C}.
    \end{align}

    On the other hand, by Lemma~\ref{lemma:RelativeSmoothnessStrongConvexitySubgradient}, we have $L\nabla \varphi^{*}(q_t)-\nabla G^f(q_t)\in\partial (L\varphi^{*}-G^f)(q_t)$ and $\nabla G^f(q_t)-\mu\nabla \varphi^{*}(q_t)\in\partial (G^f-\mu \varphi^{*})(q_t)$.
    This implies that
    \begin{align}
        G^f(q_{t+1})-G^f(q_t)
        \leq \inner{\nabla G^f(q_t)}{q_{t+1}-q_t}+LD_{\varphi^{*}}(q_{t+1},q_t),
    \end{align}
    and
    \begin{align}
        {G^f}( q)-{G^f}(q_t)
        \geq \inner{\nabla G^f(q_t)}{q-q_t}+\mu D_{\varphi^{*}}(q,q_t),\quad\text{for all } q\in\mathcal{C}.
    \end{align}
    Combining the above two inequalities with the equality~\eqref{eq:BregmanProximalInequalityMGMLemma}, we obtain
    \begin{align}
        \label{ineq:DecentLemma}
        G^f(q_{t+1})-G^f(q)
        &\leq (L-\mu)D_{\varphi^{*}}(q,q_t)-LD_{\varphi^{*}}(q,q_{t+1}),\quad\text{for all } q\in\mathcal{C}.
    \end{align}

    For the linear rate $O((1-\mu/L)^T)$, by rearranging the inequality~\eqref{ineq:DecentLemma}, we obtain
    \begin{align}
        D_{\varphi^{*}}(q_{\star},q_{t+1})
        \leq \frac{L-\mu}{L}D_{\varphi^{*}}(q_{\star},q_t)+\frac{G^f(q_{\star})-G^f(q_{t+1})}{L}
        \leq  \frac{L-\mu}{L}D_{\varphi^{*}}(q_{\star},q_t).
    \end{align}
    This implies that $D_{\varphi^{*}}(q_{\star},q_{T+1})$ decreases at a rate of $O((1-\mu/L)^T)$.
    Again, from the inequality~\eqref{ineq:DecentLemma}, we obtain
    \begin{align}
        \label{ineq:ConjugateLinearRateProof}
        G^f(q_{T+1})-G^f(q_{\star})+LD_{\varphi^{*}}(q_{\star},q_{T+1})
        &\leq (L-\mu)D_{\varphi^{*}}(q_{\star},q_T)\nonumber\\
        &\leq \left(1-\frac{\mu}{L}\right)^{T-1}(L-\mu)D_{\varphi^{*}}(q_{\star},q_1)\nonumber\\
        &=L\left(1-\frac{\mu}{L}\right)^{T}D_{\varphi^{*}}(q_{\star},q_1).
    \end{align}

    For the sublinear rate $O(1/T)$, by taking $q=q_{t}$ in the inequality~\eqref{ineq:DecentLemma}, we obtain that $G^f(q_t)$ is non-increasing in $t$.
    By telescoping the inequality~\eqref{ineq:DecentLemma} over $t$, we obtain
    \begin{align}
        \sum_{t=1}^{T}\left(G^f(q_{t+1})-G^f(q_{\star})\right)
        &\leq \sum_{t=1}^{T} \left( (L-\mu)D_{\varphi^{*}}(q_{\star},q_t)-LD_{\varphi^{*}}(q_{\star},q_{t+1}) \right)\nonumber\\
        &\leq \sum_{t=1}^{T} \left( LD_{\varphi^{*}}(q_{\star},q_t)-LD_{\varphi^{*}}(q_{\star},q_{t+1}) \right)\nonumber\\
        &= LD_{\varphi^{*}}(q_{\star},q_1)-LD_{\varphi^{*}}(q_{\star},q_{T+1})\nonumber\\
        &\leq LD_{\varphi^{*}}(q_{\star},q_1).
    \end{align}
    Dividing both sides by $T$ and using the fact that $G^f(q_{t})$ is non-increasing in $t$, we obtain
    \begin{align}
        \label{ineq:ConjugateSublinearRateProof}
        G^f(q_{T+1})-G^f(q_{\star})
        &\leq \frac{LD_{\varphi^{*}}(q_{\star},q_1)}{T}.
    \end{align}
    This concludes the proof.
\end{proof}

\subsection{Proof of Lemma~\ref{lemma:EquivalenceOpt}}
\label{appendix:ProofEquivalenceOpt}

\begin{proof}
    Let $x_{\star}$ be any optimal solution of the optimization problem~\eqref{def:GeneralProblem}.

    We first consider the case $\theta=-1$ and $0\leq\mu\leq L\leq 1$.

    Assume that
    \begin{align}
        \label{ineq:Case1PhiCvxFirstOrderOpt}
        f(x)
        \leq f(x_{\star})+\RenyiKL{x_{\star}}{x},\quad\text{for all } x\in\mathrm{ri}(\Delta_m).
    \end{align}
    Then, we have
    \begin{align}
        G^{f}(x_{\star})
        &=\sup_{y\in\mathbb{R}^m}\inner{x_{\star}}{y}+f(\exp(y))-h(x_{\star})\nonumber\\
        &=\sup_{x\in\mathbb{R}_{++}^m}\inner{x_{\star}}{\log(x)}+f(x)-h(x_{\star})\nonumber\\
        &=\sup_{x\in\mathrm{ri}(\Delta_m)}f(x)-\RenyiKL{x_{\star}}{x}\nonumber\\
        &\leq\sup_{x\in\mathrm{ri}(\Delta_m)}f(x_{\star})+\RenyiKL{x_{\star}}{x}-\RenyiKL{x_{\star}}{x}\nonumber\\
        &=f(x_{\star}),
    \end{align}
    where $h(x)=\inner{x}{\log(x)}$, and the third equality follows from the log-homogeneity of $f$ and the fact that $x_{\star}\in\Delta_m$.
    On the other hand, we also have
    \begin{align}
        \label{ineq:GqGeqfqCase1}
        G^{f}(q)
        &=\sup_{x\in\mathrm{ri}(\Delta_m)}f(x)-\RenyiKL{q}{x}\nonumber\\
        &\geq f(q)-\RenyiKL{q}{q}\nonumber\\
        &=f(q)\nonumber\\
        &\geq f(x_{\star}),\quad\text{for all } q\in\Delta_m.
    \end{align}
    This implies that $x_{\star}$ is also a minimizer of $G^f$ over $\Delta_m$ and
    \begin{align}
        \label{eq:OptEqCase1}
        G^{f}(x_{\star})
        =f(x_{\star}),
    \end{align}
    provided that the inequality~\eqref{ineq:Case1PhiCvxFirstOrderOpt} holds.

    To prove the inequality~\eqref{ineq:Case1PhiCvxFirstOrderOpt}, from the convexity of $\varphi(y)$ established in Lemma~\ref{lemma:ConjugateFunctionWellDefine}, we have
    \begin{align}
        \label{ineq:Case1PhiCvxFirstOrder}
        f(y)
        &=-\varphi(\log(y))\nonumber\\
        &\leq -\varphi(\log(x))+\inner{-\nabla \varphi(\log(x))}{\log(y)-\log(x)}\nonumber\\
        &=f(x)+\inner{-x\odot \nabla f(x)}{\log(x)-\log(y)}\nonumber\\
        &=f(x)+\RenyiKL{-x\odot \nabla f(x)}{y}-\RenyiKL{-x\odot \nabla f(x)}{x},\quad\text{for all } x,y\in\mathrm{ri}(\Delta_m).
    \end{align}

    To address the fact that $x_{\star}$ may not lie in $\mathbb{R}_{++}^m$, we define a sequence $\Set{x_{\star,k}}_{k=1}^{\infty}\subset\mathrm{ri}(\Delta_m)$ such that $x_{\star,k}=(1-1/k)x_{\star}+\mathbf{1}_m/(mk)$, where we have $\lim_{k\to\infty}x_{\star,k}=x_{\star}$.
    Then, from the inequality~\eqref{ineq:Case1PhiCvxFirstOrder}, we have
    \begin{align}
        0
        &\leq \RenyiKL{-x_{\star,k}\odot\nabla f(x_{\star,k})}{x_{\star,k}}\nonumber\\
        &\leq f(x_{\star,k})-f(-x_{\star,k}\odot\nabla f(x_{\star,k}))\nonumber\\
        &\leq f(x_{\star,k})-f(x_{\star}),\nonumber
    \end{align}
    where the last inequality follows from the optimality of $x_{\star}$ and the fact that $-x_{\star,k}\odot\nabla f(x_{\star,k})\in\Delta_m$ due to Lemma~\ref{lemma:SelfNormalize}.
    Since $\lim_{k\to\infty}f(x_{\star,k})-f(x_{\star})=0$ by the continuity of $f$, we have $\lim_{k\to\infty}\RenyiKL{-x_{\star,k}\odot\nabla f(x_{\star,k})}{x_{\star,k}}=0$.
    This implies that 
    \begin{align}
        \lim_{k\to\infty}-x_{\star,k}\odot\nabla f(x_{\star,k})
        =x_{\star}.
    \end{align}
    Applying the inequality~\eqref{ineq:Case1PhiCvxFirstOrder} again, we obtain
    \begin{align}
        f(x)
        \leq f(x_{\star,k})+\RenyiKL{-x_{\star,k}\odot\nabla f(x_{\star,k})}{x}-\RenyiKL{-x_{\star,k}\odot\nabla f(x_{\star,k})}{x_{\star,k}},
    \end{align}
    for all $x\in\mathrm{ri}(\Delta_m)$.
    The inequality~\eqref{ineq:Case1PhiCvxFirstOrderOpt} then follows by taking the limit $k\to\infty$ in the above inequality.

    Next, we consider the case $\theta=-1$, $\mu\geq 1$, and $L>1$.

    Assume that
    \begin{align}
        \label{ineq:Case2PhiCvxFirstOrderOpt}
        f(x)
        \geq f(x_{\star})+\RenyiKL{x_{\star}}{x},\quad\text{for all } x\in\mathrm{ri}(\Delta_m).
    \end{align}
    Then, we have
    \begin{align}
        G^f(x_{\star})
        &=\sup_{y\in\mathbb{R}^m}h(x_{\star})+\inner{x_{\star}}{y}-f(\exp(-y))\nonumber\\
        &=\sup_{x\in\mathrm{ri}(\Delta_m)}\RenyiKL{x_{\star}}{x}-f(x)\nonumber\\
        &\leq\sup_{x\in\mathrm{ri}(\Delta_m)}\RenyiKL{x_{\star}}{x}-f(x_{\star})-\RenyiKL{x_{\star}}{x}\nonumber\\
        &=-f(x_{\star}).
    \end{align}
    On the other hand, we also have
    \begin{align}
        G^f(q)
        &=\sup_{x\in\mathrm{ri}(\Delta_m)}\RenyiKL{q}{x}-f(x)\nonumber\\
        &\geq \RenyiKL{q}{x_{\star}}-f(x_{\star})\nonumber\\
        &\geq-f(x_{\star}),\quad\text{for all } q\in\Delta_m.\nonumber
    \end{align}
    This implies that $x_{\star}$ is also a minimizer of $G^f$ over $\Delta_m$, and
    \begin{align}
        \label{eq:OptEqCase2}
        G^f(x_{\star})
        =-f(x_{\star}),
    \end{align}
    provided that the inequality~\eqref{ineq:Case2PhiCvxFirstOrderOpt} holds.

    To prove the inequality~\eqref{ineq:Case2PhiCvxFirstOrderOpt}, from the convexity of $\varphi(y)$ and an argument similar to that for the inequality~\eqref{ineq:Case1PhiCvxFirstOrder}, we obtain that
    \begin{align}
        \label{ineq:Case2PhiCvxFirstOrder}
        f(y)
        &\geq f(x)+\RenyiKL{-x\odot \nabla f(x)}{y}-\RenyiKL{-x\odot \nabla f(x)}{x},\quad\text{for all } x,y\in\mathrm{ri}(\Delta_m).
    \end{align}

    By taking $x=\mathbf{1}_m/m$ in the inequality~\eqref{ineq:Case2PhiCvxFirstOrder}, we obtain
    \begin{align}
        f(y)
        &\geq f(\mathbf{1}_m/m)+\inner{-\frac{1}{m}\nabla f\left(\frac{\mathbf{1}_m}{m}\right)}{-\log(y)}+\text{constant},\quad\text{for all } y\in\mathrm{ri}(\Delta_m).
    \end{align}
    From the assumption that $-\nabla f(x)\in\mathbb{R}_{++}^m$ and that $f$ is continuous on $\mathbb{R}_{+}^m\backslash\{0\}$, we obtain that $f(y)=\infty$ for all $y\in\Delta_m\setminus\mathrm{ri}(\Delta_m)$.
    This implies that $x_{\star}\in\mathrm{ri}(\Delta_m)$, provided that $\mu\geq 1$ and $L>1$.

    Therefore, by the optimality condition for $x_{\star}$, we have $\nabla f(x_{\star})=\lambda \mathbf{1}_m$ for some scalar $\lambda$.
    Consequently, from Lemma~\ref{lemma:SelfNormalize}, we have $x_{\star}\odot (-\nabla f(x_{\star})) = x_{\star}$.
    This equality, combined with the inequality~\eqref{ineq:Case2PhiCvxFirstOrder}, yields the inequality~\eqref{ineq:Case2PhiCvxFirstOrderOpt}.

    For the case $\theta=1$, we write
    \begin{align}
        G^f(q)
        &=\inf_{y\in\mathbb{R}^m}h(q)-\inner{q}{y}+f(\exp(y))\nonumber\\
        &=\inf_{x\in\mathbb{R}_{++}^m}\inner{q}{\log(q)-\log(x)}+f(x)\nonumber\\
        &=\inf_{x\in\mathrm{ri}(\Delta_m)}\RenyiKL{q}{x}+f(x)\nonumber\\
        &\geq\inf_{x\in\mathrm{ri}(\Delta_m)}0+f(x_{\star})\nonumber\\
        &= f(x_{\star}),\quad\text{for all } q\in\Delta_m.\nonumber
    \end{align}
    On the other hand, we also have
    \begin{align}
        G^f(x_{\star})
        &=\inf_{x\in\mathrm{ri}(\Delta_m)}\RenyiKL{x_{\star}}{x}+f(x)\nonumber\\
        &\leq \RenyiKL{x_{\star}}{x_{\star,k}}+f(x_{\star,k}).\nonumber
    \end{align}
    Taking the limit as $k\to\infty$ gives $G^f(x_{\star})\leq f(x_{\star})$, and thus $x_{\star}$ is also a minimizer of $G^f$ over $\Delta_m$, and
    \begin{align}
        \label{eq:OptEqCase3}
        G^f(x_{\star})
        =f(x_{\star}).
    \end{align}
    
    Finally, from the equalities \eqref{eq:OptEqCase1}, \eqref{eq:OptEqCase2}, and \eqref{eq:OptEqCase3}, we have $-\infty<G^f(x_{\star})=c_f\theta f(x_{\star})<\infty$.
    This implies that $\varphi^{*}(x_{\star})\leq \abs{G^f(x_{\star})}+\abs{\inner{x_{\star}}{\log(x_{\star})}}<\infty$.
    This concludes the proof.
\end{proof}

\subsection{Proof of Theorem~\ref{theorem:ConvergenceRateMGM}}
\label{appendix:ProofConvergenceRateMGM}
The proof of Theorem~\ref{theorem:ConvergenceRateMGM} is based on Lemma~\ref{lemma:FullDomainForThetaMinusOne}.

\begin{lemma}
    \label{lemma:FullDomainForThetaMinusOne}
    Let the function $f$ be defined as in Theorem~\ref{theorem:ConvergenceRateMGM}.
    Then, we have $\mathcal{C}=\Delta_m$ for the case $\theta=-1$ and $L<1$, as well as the case $\theta=-1$ and $\mu>1$, where $\mathcal{C}$ is defined in Lemma~\ref{lemma:RelativeSmoothnessStrongConvexity}.
\end{lemma}
\begin{proof}
    For $L<1$,
    we have
    \begin{align}
        \varphi^{*}(q)
        &=\sup_{y\in\mathbb{R}^m}\inner{q}{y}+f(\exp(y))\nonumber\\
        &=\sup_{z\in\mathbb{R}_{++}^m}\frac{\inner{q}{\log(z)}-\log((L-1)U_{-1,L}^f(z))}{1-L}\nonumber\\
        &\leq\sup_{z\in\mathbb{R}_{++}^m}\frac{\inner{q}{\log(z)}-\log(\inner{(L-1)\nabla U_{-1,L}^f(\mathbf{1}_m)}{z})}{1-L}\nonumber\\
        &=\frac{\inner{q}{\log(q)-\log((L-1)\nabla U_{-1,L}^f(\mathbf{1}_m))}}{1-L}\nonumber\\
        &<\infty,\quad\text{for all } q\in\Delta_m,\nonumber
    \end{align}
    where the first inequality follows from the convexity of $(L-1)U_{-1,L}^f$ established in Lemma~\ref{lemma:ZeroOrderCharacterization} and the fact that 
    \begin{align}
        (L-1)U_{-1,L}^f(z)
        &\geq (L-1)U_{-1,L}^f(\mathbf{1}_m)+\inner{(L-1)\nabla U_{-1,L}^f(\mathbf{1}_m)}{z-\mathbf{1}_m}\nonumber\\
        &= \inner{(L-1)\nabla U_{-1,L}^f(\mathbf{1}_m)}{z},\nonumber
    \end{align}
    and the last inequality follows from the assumption that $\theta\nabla f(x)>0$ for all $x\in\mathbb{R}_{++}^m$.

    Similarly, for $\mu>1$, we have
    \begin{align}
        \varphi^{*}(q)
        &=\sup_{y\in\mathbb{R}^m}\inner{q}{y}-f(\exp(-y))\nonumber\\
        &=\sup_{z\in\mathbb{R}_{++}^m}\frac{\inner{q}{\log(z)}-\log((\mu-1)U_{-1,\mu}^f(z))}{\mu-1}\nonumber\\
        &\leq\sup_{z\in\mathbb{R}_{++}^m}\frac{\inner{q}{\log(z)}-\log(\inner{(\mu-1)\nabla U_{-1,\mu}^f(\mathbf{1}_m)}{z})}{\mu-1}\nonumber\\
        &=\frac{\inner{q}{\log(q)-\log((\mu-1)\nabla U_{-1,\mu}^f(\mathbf{1}_m))}}{\mu-1}\nonumber\\
        &<\infty,\quad\text{for all } q\in\Delta_m.\nonumber
    \end{align}
    This concludes the proof.
\end{proof}

\begin{proofof}{Theorem~\ref{theorem:ConvergenceRateMGM}}

    For the case $\theta=-1$ and $L\leq 1$, we will use the bounds on $G^f(q)-G^f(q_t)$ with $q=x_{t+1}$ 
    obtained from the property $Lr_t-g_t\in\partial (L\varphi^*-G^f)(q_t)$ established in Lemma~\ref{lemma:RelativeSmoothnessStrongConvexitySubgradient}, for which we have to first verify that $x_{t+1}\in\mathcal{C}$. 
    The claim for $L<1$ is already established in Lemma~\ref{lemma:FullDomainForThetaMinusOne}.
    For $L=1$, we have $x_{t+1}=q_t=\nabla \varphi(\log(x_t))$, which, by Lemma~\ref{lemma:FenchelSubgradient}, implies that $x_{t+1}\in\mathcal{C}$ and $\varphi^{*}(x_{t+1})=\varphi^{*}(q_t)=\inner{q_t}{\log(x_t)}+f(x_t)$.

    By applying the subgradient inequality for $Lr_t-g_t\in\partial (L\varphi^*-G^f)(q_t)$ from Lemma~\ref{lemma:RelativeSmoothnessStrongConvexitySubgradient} at $q=x_{t+1}$, we have
    \begin{align}
        &G^f(x_{t+1})-G^f(q_t)\nonumber\\
        &\quad=G^f(x_{t+1})-(\varphi^{*}(q_t)-\inner{q_t}{\log(q_t)})\nonumber\\
        &\quad=G^f(x_{t+1})-(\inner{q_t}{\log(x_t)}+f(x_t)-\inner{q_t}{\log(q_t)})\nonumber\\
        &\quad\leq L D_{\varphi^{*}}(x_{t+1},q_t)+\inner{g_t}{x_{t+1}-q_t}\nonumber\\
        &\quad=L(\varphi^{*}(x_{t+1})-\varphi^{*}(q_t)-\inner{\log(x_t)}{x_{t+1}-q_t})\nonumber\\
        &\quad\quad+\inner{\log(x_t)-\log(q_t)-\mathbf{1}_m}{x_{t+1}-q_t}\nonumber\\
        &\quad=L(G^f(x_{t+1})+\inner{x_{t+1}}{\log(x_{t+1})}-\inner{q_t}{\log(x_t)}-f(x_t)-\inner{\log(x_t)}{x_{t+1}-q_t})\nonumber\\
        &\quad\quad+\inner{\log(x_t)-\log(q_t)}{x_{t+1}-q_t},
    \end{align}
    where the second and last equalities follow from Lemma~\ref{lemma:FenchelSubgradient}.
    By rearranging the above inequality, we obtain
    \begin{align}
        (1-L)G^f(x_{t+1})
        \leq (1-L)f(x_t)+\RenyiKL{x_{t+1}}{q_t}-(1-L)\RenyiKL{x_{t+1}}{x_t}.
    \end{align}
    Furthermore, by the inequality~\eqref{ineq:GqGeqfqCase1}, we have
    \begin{align}
        &(1-L)f(x_{t+1})\nonumber\\
        &\quad\leq (1-L)G^f(x_{t+1})\nonumber\\
        &\quad\leq  (1-L)f(x_t)+\RenyiKL{x_{t+1}}{q_t}-(1-L)\RenyiKL{x_{t+1}}{x_t}\nonumber\\
        &\quad =(1-L)f(x_t)+\inner{x_{t+1}}{\log\left((-\nabla f(x_t))^{1/L-1}\right)-\log\left(\norm{x_t\odot(-\nabla f(x_t))^{1/L}}_1\right)\mathbf{1}_m}\nonumber\\
        &\quad\quad-(1-L)\inner{x_{t+1}}{\log\left((-\nabla f(x_t))^{1/L}\right)-\log\left(\norm{x_t\odot(-\nabla f(x_t))^{1/L}}_1\right)\mathbf{1}_m}\nonumber\\
        &\quad =(1-L)f(x_t)-\log\left(\norm{x_t\odot(-\nabla f(x_t))^{1/L}}_1^{L}\right)\label{ineq:IterateNormBoundCase1}\\
        &\quad \leq (1-L)f(x_t)-\log\left(\frac{\norm{x_t\odot(-\nabla f(x_t))}_1}{\norm{x_t}_1^{1-L}}\right)\nonumber\\
        &\quad= (1-L)f(x_t),
    \end{align}
    where the first inequality follows from the inequality~\eqref{ineq:GqGeqfqCase1}, and the last inequality follows from H\"{o}lder's inequality.
    This implies that $f(x_t)$ is non-increasing for the case $\theta=-1$ and $L< 1$.
    In addition, the monotonicity of $f(x_t)$ for the case $\theta=-1$ and $L=1$ can also be verified by writing $f(x_{t+1})=f(q_t)\leq G^f(q_t)=f(x_t)-\RenyiKL{q_t}{x_t}\leq f(x_t)$.

    By the lower bound on $G^f(q)-G^f(q_t)$ with $q=x_{\star}$ obtained from the property $g_t-\mu r_t\in\partial (G^f-\mu \varphi^*)(q_t)$ in Lemma~\ref{lemma:RelativeSmoothnessStrongConvexitySubgradient}, we have
    \begin{align}
        &G^f(x_{\star})-G^f(q_t)\nonumber\\
        &\quad=f(x_{\star})-G^f(q_t)\nonumber\\
        &\quad=f(x_{\star})-(\varphi^{*}(q_t)-\inner{q_t}{\log(q_t)})\nonumber\\
        &\quad=f(x_{\star})-(\inner{q_t}{\log(x_t)}+f(x_t)-\inner{q_t}{\log(q_t)})\nonumber\\
        &\quad\geq \mu D_{\varphi^{*}}(x_{\star},q_t)+\inner{g_t}{x_{\star}-q_t}\nonumber\\
        &\quad=\mu(\varphi^{*}(x_{\star})-\varphi^{*}(q_t)-\inner{\log(x_t)}{x_{\star}-q_t})+\inner{g_t}{x_{\star}-q_t}\nonumber\\
        &\quad=\mu(G^f(x_{\star})+\inner{x_{\star}}{\log(x_{\star})}-\inner{q_t}{\log(x_t)}-f(x_t)-\inner{\log(x_t)}{x_{\star}-q_t})\nonumber\\
        &\quad\quad+\inner{\log(x_t)-\log(q_t)}{x_{\star}-q_t}\nonumber\\
        &\quad=\mu(f(x_{\star})+\inner{x_{\star}}{\log(x_{\star})}-f(x_t)-\inner{\log(x_t)}{x_{\star}})\nonumber\\
        &\quad\quad+\inner{\log(x_t)-\log(q_t)}{x_{\star}-q_t},\nonumber
    \end{align}
    where the first equality follows from the equality~\eqref{eq:OptEqCase1}.
    By rearranging the above inequality, we obtain
    \begin{align}
        &(1-\mu)f(x_{\star})+(1-\mu)\RenyiKL{x_{\star}}{x_t}\nonumber\\
        &\quad\geq (1-\mu)f(x_t)+\RenyiKL{x_{\star}}{q_t}\nonumber\\
        &\quad =(1-\mu)f(x_t)+\inner{x_{\star}}{\log(x_{\star})-\log(x_{t}\odot(-\nabla f(x_t)))}\nonumber\\
        &\quad =(1-\mu)f(x_t)+\inner{x_{\star}}{\log(x_{\star})-L\log(x_{t}\odot(-\nabla f(x_t))^{1/L})-(1-L)\log(x_t)}\nonumber\\
        &\quad =(1-\mu)f(x_t)+\inner{x_{\star}}{\log(x_{\star})-L\log(x_{t+1})-(1-L)\log(x_t)}\nonumber\\
        &\quad\quad -L\log\left(\norm{x_t\odot(-\nabla f(x_t))^{1/L}}_1\right)\nonumber\\
        &\quad=(1-\mu)f(x_t)+L\RenyiKL{x_{\star}}{x_{t+1}}+(1-L)\RenyiKL{x_{\star}}{x_t}\nonumber\\
        &\quad\quad -L\log\left(\norm{x_t\odot(-\nabla f(x_t))^{1/L}}_1\right).
    \end{align}
    Combining the above inequality with the inequality~\eqref{ineq:IterateNormBoundCase1} yields
    \begin{align}
        &\frac{1-\mu}{L}(f(x_t)-f(x_{\star}))+\RenyiKL{x_{\star}}{x_{t+1}}\nonumber\\
        &\quad\leq \left(1-\frac{\mu}{L}\right)\RenyiKL{x_{\star}}{x_{t}}+\log\left(\norm{x_t\odot(-\nabla f(x_t))^{1/L}}_1\right)\nonumber\\
        &\quad\leq \left(1-\frac{\mu}{L}\right)\RenyiKL{x_{\star}}{x_{t}}+\left(\frac{1}{L}-1\right)(f(x_{t})-f(x_{t+1})),
    \end{align}
    which is equivalent to the inequality
    \begin{align}
        \label{ineq:RenyiKLInductionCase1}
        &\frac{1}{L}(f(x_{t+1})-f(x_{\star}))+(\RenyiKL{x_{\star}}{x_{t+1}}-(f(x_{t+1})-f(x_{\star})))\nonumber\\
        &\quad \leq \left(1-\frac{\mu}{L}\right)(\RenyiKL{x_{\star}}{x_{t}}-(f(x_{t})-f(x_{\star}))).
    \end{align}
    By induction, we obtain
    \begin{align}
        &\frac{1}{L}(f(x_{T+1})-f(x_{\star}))\nonumber\\
        &\quad\leq \frac{1}{L}(f(x_{T+1})-f(x_{\star}))+D_{\varphi^{*}}(x_{\star},q_{T+1})\nonumber\\
        &\quad=\frac{1}{L}(f(x_{T+1})-f(x_{\star}))+(\RenyiKL{x_{\star}}{x_{T+1}}-(f(x_{T+1})-f(x_{\star})))\nonumber\\
        &\quad \leq \left(1-\frac{\mu}{L}\right)^{T}(\RenyiKL{x_{\star}}{x_{1}}-(f(x_{1})-f(x_{\star})))\nonumber\\
        &\quad = \left(1-\frac{\mu}{L}\right)^{T}D_{\varphi^{*}}(x_{\star},q_{1}).\nonumber
    \end{align}
    Therefore, we obtain that $(1-\mu/L)^T\RenyiKL{x_{\star}}{x_{1}}$ is an upper bound on both $(f(x_{T+1})-f(x_{\star}))/L$ and $\RenyiKL{x_{\star}}{x_{T+1}}$.

    Also, by telescoping the inequality~\eqref{ineq:RenyiKLInductionCase1}, we have
    \begin{align}
        &f(x_{T+1})-f(x_{\star})\nonumber\\
        &\quad\leq \frac{1}{T}\sum_{t=1}^{T} (f(x_{t+1})-f(x_{\star}))\nonumber\\
        &\quad\leq \frac{1}{T}\sum_{t=1}^{T} \left((L-\mu)D_{\varphi^{*}}(x_{\star},q_{t})-LD_{\varphi^{*}}(x_{\star},q_{t+1})\right)\nonumber\\
        &\quad\leq \frac{1}{T}\sum_{t=1}^{T} \left(LD_{\varphi^{*}}(x_{\star},q_{t})-LD_{\varphi^{*}}(x_{\star},q_{t+1})\right)\nonumber\\
        &\quad\leq \frac{1}{T}\left(LD_{\varphi^{*}}(x_{\star},q_{1})-LD_{\varphi^{*}}(x_{\star},q_{T+1})\right)\nonumber\\
        &\quad\leq \frac{1}{T}LD_{\varphi^{*}}(x_{\star},q_{1})\nonumber\\
        &\quad\leq \frac{1}{T}\RenyiKL{x_{\star}}{x_{1}}.\nonumber
    \end{align}
    This concludes the proof of Theorem~\ref{theorem:ConvergenceRateMGM} for the case $\theta=-1$ and $L\leq 1$.

    For the case $\theta=-1$, $1\leq \mu\leq L$ and $L>1$, from the upper bound on $G^f(q)-G^f(q_{t})$ established in Lemma~\ref{lemma:RelativeSmoothnessStrongConvexitySubgradient} with $q=q_{t+1}$, we have
    \begin{align}
        &G^f(q_{t+1})-\RenyiKL{q_{t}}{x_{t}} +f(x_{t})\nonumber\\
        &\quad=G^f(q_{t+1})-\inner{q_{t}}{-\log(x_{t})}+\varphi(-\log(x_t))-\inner{q_t}{\log(q_t)}\nonumber\\
        &\quad=G^f(q_{t+1})-\varphi^{*}(q_t)-\inner{q_t}{\log(q_t)}\nonumber\\
        &\quad= G^f(q_{t+1})-G^f(q_{t})\nonumber\\
        &\quad\leq LD_{\varphi^{*}}(q_{t+1},q_t)+\inner{g_t}{q_{t+1}-q_t}\nonumber\\
        &\quad=L(\varphi^{*}(q_{t+1})-\varphi^{*}(q_t)-\inner{-\log(x_{t})}{q_{t+1}-q_t})+\inner{-\log(x_t)+\log(q_t)}{q_{t+1}-q_t}\nonumber\\
        &\quad=L(\inner{q_{t+1}}{-\log(x_{t+1})}-\varphi(-\log(x_{t+1})))-L(\inner{q_{t}}{-\log(x_{t})}-\varphi(-\log(x_t)))\nonumber\\
        &\quad\quad-L\inner{-\log(x_{t})}{q_{t+1}-q_t}+\inner{-\log(x_t)+\log(q_t)}{q_{t+1}-q_t}\nonumber\\
        &\quad=Lf(x_t)-Lf(x_{t+1})-\RenyiKL{q_t}{x_t}\nonumber\\
        &\quad\quad+L\inner{q_{t+1}}{\log(x_t)-\log(x_{t+1})}+\inner{-\log(x_t)+\log(q_t)}{q_{t+1}}\nonumber\\
        &\quad=Lf(x_t)-Lf(x_{t+1})-\RenyiKL{q_t}{x_t}\nonumber\\
        &\quad\quad+L\inner{q_{t+1}}{\log(x_t)-\log\left(\frac{x_t\odot\left(-\nabla f(x_t)\right)^{1/L}}{\norm{x_t\odot\left(-\nabla f(x_t)\right)^{1/L}}_1}\right)}+\inner{-\log(x_t)+\log(q_t)}{q_{t+1}}\nonumber\\
        &\quad=Lf(x_t)-Lf(x_{t+1})-\RenyiKL{q_t}{x_t}+L\log\left(\norm{x_t\odot\left(-\nabla f(x_t)\right)^{1/L}}_1\right),\nonumber
    \end{align}
    where the second and fifth equalities follow from Lemma~\ref{lemma:FenchelSubgradient}.
    By rearrangement, we obtain
    \begin{align}
        0
        &\leq  G^f(q_{t+1})-G^f(x_{\star})\nonumber\\
        &= G^f(q_{t+1})+f(x_{\star})\nonumber\\
        &\leq(L-1)(f(x_t)-f(x_{\star}))-L(f(x_{t+1})-f(x_{\star}))+L\log\left(\norm{x_t\odot\left(-\nabla f(x_t)\right)^{1/L}}_1\right)\label{ineq:IterateNormBoundCase2}\\
        &\leq (L-1)(f(x_t)-f(x_{\star}))-L(f(x_{t+1})-f(x_{\star}))\nonumber\\
        &\quad+L\log\left(\norm{x_t\odot\left(-\nabla f(x_t)\right)}_1^{1/L}\cdot\norm{x_t}_1^{1-1/L}\right)\nonumber\\
        &= (L-1)(f(x_t)-f(x_{\star}))-L(f(x_{t+1})-f(x_{\star})),\nonumber
    \end{align}
    where the first equality follows from the equality~\eqref{eq:OptEqCase2}, and the last inequality follows from H\"{o}lder's inequality.
    For the case $\mu=1$, combining the above inequality and the fact that $D_{\varphi^{*}}(x_{\star},q_{t})=f(x_t)-f(x_{\star})-\RenyiKL{x_{\star}}{x_t}\geq 0$, we obtain that $\RenyiKL{x_{\star}}{x_{T+1}}\leq f(x_{T+1})-f(x_{\star})\leq (1-\mu/L)^T(f(x_1)-f(x_{\star}))$.

    For the case $\mu>1$, we further use Lemma~\ref{lemma:FullDomainForThetaMinusOne} and the lower bound on $G^f(q)-G^f(q_{t})$  obtained from the property $g_t-\mu r_t\in\partial (G^f-\mu\varphi^*)(q_t)$  to obtain
    \begin{align}
        &G^f(q)+f(x_t)-\RenyiKL{q_t}{x_t}\nonumber\\
        &\quad= G^f(q)-G^f(q_{t})\nonumber\\
        &\quad\geq \mu D_{\varphi^{*}}(q,q_{t})+\inner{g_t}{q-q_t}\nonumber\\
        &\quad =\mu(G^f(q)+f(x_t)-\RenyiKL{q}{x_t})+\inner{-\log(x_t)+\log(q_t)}{q-q_t},\quad\text{for all } q\in\Delta_m.
    \end{align}
    This implies that
    \begin{align}
        &-f(x_{\star})+f(x_t)\nonumber\\
        &\quad=G^f(x_{\star})+f(x_t)\nonumber\\
        &\quad=\inf_{q\in\Delta_m} G^f(q)+f(x_t)\nonumber\\
        &\quad\leq \inf_{q\in\Delta_m}\RenyiKL{q}{x_t}+\frac{\RenyiKL{q}{q_t}}{\mu-1}\nonumber\\
        &\quad= \frac{-L}{\mu-1}\log\left(\norm{x_t\odot(-\nabla f(x_t))^{1/\mu}}_1^{\mu/L}\right)\nonumber\\
        &\quad\leq \frac{-L}{\mu-1}\log\left(\norm{x_t\odot(-\nabla f(x_t))^{1/L}}_1\cdot\norm{x_t}_1^{\mu/L-1}\right)\nonumber\\
        &\quad= \frac{-L}{\mu-1}\log\left(\norm{x_t\odot(-\nabla f(x_t))^{1/L}}_1\right),\nonumber
    \end{align}
    where the first and second equalities follow from the equality~\eqref{eq:OptEqCase2}, and the last inequality follows from H\"{o}lder's inequality.
    Combining the above inequality with the inequality~\eqref{ineq:IterateNormBoundCase2}, as well as with the fact that $D_{\varphi^{*}}(x_{\star},q_{t})=f(x_t)-f(x_{\star})-\RenyiKL{x_{\star}}{x_t}\geq 0$, we obtain 
    \begin{align}
        \RenyiKL{x_{\star}}{x_{T+1}}
        \leq f(x_{T+1})-f(x_{\star})
        \leq \left(1-\frac{\mu}{L}\right)^{T}(f(x_1)-f(x_{\star})).
    \end{align}
    This concludes the proof for the case $\theta=-1$, $1\leq \mu\leq L$ and $L>1$.

    Finally, for the case $\theta=1$, from the inequality~\eqref{ineq:ConjugateLinearRateProof}, we have
    \begin{align}
        &(L+1)(f(x_{T+1})-f(x_{\star}))+L\RenyiKL{x_{\star}}{x_{T+1}}\nonumber\\
        &\quad\leq \RenyiKL{q_{T+1}}{x_{T+1}}+(L+1)(f(x_{T+1})-f(x_{\star}))+L\RenyiKL{x_{\star}}{x_{T+1}}\nonumber\\
        &\quad=G^f(q_{T+1})-G^f(x_{\star})+LD_{\varphi^{*}}(x_{\star},q_{T+1})\nonumber\\
        &\quad\leq L\left(1-\frac{\mu}{L}\right)^{T}D_{\varphi^{*}}(x_{\star},q_{1})\nonumber\\
        &\quad=L\left(1-\frac{\mu}{L}\right)^{T}(f(x_1)-f(x_{\star})+\RenyiKL{x_{\star}}{x_1}),
    \end{align}
    where the first and last equalities follow from Lemma~\ref{lemma:FenchelSubgradient} and the equality~\eqref{eq:OptEqCase3}.
    This proves the desired linear rate.

    Similarly, from the inequality~\eqref{ineq:ConjugateSublinearRateProof}, we also have
    \begin{align}
        &f(x_{T+1})-f(x_{\star})\nonumber\\
        &\quad\leq\RenyiKL{q_{T+1}}{x_{T+1}}+f(x_{T+1})-f(x_{\star})\nonumber\\
        &\quad=G^f(q_{T+1})-G^f(x_{\star})\nonumber\\
        &\quad\leq\frac{LD_{\varphi^{*}}(x_{\star},q_1)}{T}\nonumber\\
        &\quad=\frac{L(f(x_1)-f(x_{\star})+\RenyiKL{x_{\star}}{x_1})}{T}.
    \end{align}
    This concludes the proof.
\end{proofof}

\section{Applications of Our Framework}
\label{appendix:Applications}
In Appendix~\ref{appendix:FisherMarket}, Appendix~\ref{appendix:KellyPortfolio}, and Appendix~\ref{appendix:DOptimalDesign}, we apply our framework to the Fisher market equilibrium problem, Kelly's portfolio problem, and the D-optimal design problem, respectively.
In Appendix~\ref{appendix:ApplicationLewisWeight} and Appendix~\ref{appendix:ApplicationPetzRenyiCapacity}, we provide detailed proofs of the results stated in Sections~\ref{subsec:LewisWeights} and~\ref{subsec:PetzRenyiCapacity}, where we apply our framework to the $\ell_p$-Lewis weight problem and the Petz-R\'{e}nyi capacity problem, respectively.

\subsection{Application to Fisher Market Equilibrium}
\label{appendix:FisherMarket}
Let the vectors $w$ and $a_i$ be defined as in the equilibrium problem~\eqref{def:Fisher}.
Without loss of generality, we assume that $w\in\mathrm{ri}(\Delta_n)$, $a_i\neq 0$, and $\sum_{i=1}^n a_i[j]>0$ for all $1\leq j\leq m$.

The following lemma shows that our framework applies to the equilibrium problem.
\begin{lemma}
    \label{lemma:FisherCurvature}
    For all $\rho\in(-\infty,0)\cup(0,1)$, the objective function $f_{\text{Fis}}(x)$ of the optimization problem~\eqref{def:Fisher}  satisfies Assumption~\ref{assumption:BasicAssumptions} and the second-order characterization~\eqref{ineq:IntroHessianBounds} with 
    $\theta=-1$, $\mu=\min\Set{1,1/(1-\rho)}$, and $L=\max\Set{1,1/(1-\rho)}$.
\end{lemma}
The proof of Lemma~\ref{lemma:FisherCurvature} is provided in Appendix~\ref{appendix:ProofFisherCurvature}.
Furthermore, from the log-homogeneity of $f_{\text{Fis}}(x)$ and the assumption that $-\nabla f_{\text{Fis}}(x)>0$, we have $f_{\text{Fis}}(x)\geq f_{\text{Fis}}(\mathbf{1}_m)=f_{\text{Fis}}(\mathbf{1}_m/m)-\log(m)$ for all $x\in\Delta_m$.
Combining this fact with Lemma~\ref{lemma:FisherCurvature} and Theorem~\ref{theorem:ConvergenceRateMGM}, we obtain the following corollary.
\begin{corollary}
    \label{cor:FisherConv}
    Let $x_t$ be the sequence generated by the MGM~\eqref{alg:IntroMGM} for solving the optimization problem~\eqref{def:Fisher} with $\theta=-1$, $x_1=\mathbf{1}_m/m$, and $\eta=1/\max\Set{1,1/(1-\rho)}$.
    Then, for all $\rho\in(-\infty,0)\cup(0,1)$, we have
    \begin{align}
        f_{\text{Fis}}(x_{T+1})-f_{\text{Fis}}(x_{\text{Fis},\star})
        \leq \sigma^T\log\left(m\right),\nonumber
    \end{align}
    where $\sigma=\rho$ for $\rho\in(0,1)$ and $\sigma=\rho/(\rho-1)$ for $\rho\in(-\infty,0)$.
\end{corollary}

The MGM~\eqref{alg:IntroMGM} with $\eta=1$ and its guarantee in Corollary~\ref{cor:FisherConv} recover the algorithm and its convergence rate established by \citet{Cheung2018c} for $\rho\in(-\infty,0)$.

For the case where $\rho\in(-1,0)\cup(0,1)$, \citet{Chu2025} also established a convergence rate $\tilde{O}(\abs{\rho}^T)$ in terms of the optimization error bound.
Our convergence guarantee in Corollary~\ref{cor:FisherConv} recovers this result for $\rho>0$ and is even faster for $\rho<0$.

Finally, we note that, in addition to the optimization error bound, \citet{Chu2025} also established the same convergence rate in terms of the Thompson metric $\norm{\log(x)-\log(x_{\text{Fis},\star})}_{\infty}$.
As shown in Appendix~\ref{appendix:ComparisonErrorMeasures}, this metric upper bounds the optimization error. 
While we recover the optimization error bound, we do not recover their convergence guarantee in terms of the Thompson metric.

\subsubsection{Proof of Lemma~\ref{lemma:FisherCurvature}}
\label{appendix:ProofFisherCurvature}

The proof of Lemma~\ref{lemma:FisherCurvature} is based on Lemma~\ref{lemma:LogHomoHessianEq} and Lemma~\ref{lemma:RemoveRankOne}.
\begin{lemma}
    \label{lemma:LogHomoHessianEq}
    For any $\theta$-log-homogeneous function $f$ that is twice differentiable on $\mathbb{R}_{++}^m$, we have $\nabla^2 f(x)x=- \nabla f(x)$ for all $x\in\mathbb{R}_{++}^m$.
\end{lemma}
\begin{proof}
    Define $g(x)=\inner{\nabla f(x)}{x}$.
    From Lemma~\ref{lemma:SelfNormalize}, we have $g(x)=\theta$ for all $x\in\mathbb{R}_{++}^m$.
    Differentiating $g(x)$ yields
    \begin{align}
        \nabla g(x)
        =\nabla^2 f(x) x + \nabla f(x)
        =0.
    \end{align}
    This concludes the proof.
\end{proof}

\begin{lemma}
    \label{lemma:RemoveRankOne}
    Let $f$ be any $(-1)$-log-homogeneous function that satisfies Assumption~\ref{assumption:BasicAssumptions}.
    Furthermore, assume that $f$ is twice differentiable on $\mathbb{R}_{++}^m$ and satisfies the following inequalities:
    \begin{align}
        \mu\mathrm{Diag}\left(-\nabla f(x)\oslash x\right)
        \leq \nabla^2 f(x)
        \leq L\mathrm{Diag}\left(-\nabla f(x)\oslash x\right),\quad\text{for all } x\in\mathbb{R}_{++}^m.
    \end{align}
    Then, $f$ satisfies the   second-order characterization~\eqref{ineq:IntroHessianBounds} with the same constants $\mu$ and $L$.
\end{lemma}
\begin{proof}
    Define $Q(x)=I_m+x\nabla f(x)^{\top}$.
    We have
    \begin{align}
        &\mu \mathrm{Diag}(-\nabla f(x)\oslash x)-\mu\nabla f(x)\nabla f(x)^{\top}\nonumber\\
        &=\mu Q(x)^{\top}\mathrm{Diag}(-\nabla f(x)\oslash x)Q(x)\nonumber\\
        &\leq Q(x)^{\top}\nabla^2 f(x) Q(x)\nonumber\\
        &=\nabla^2 f(x)-\nabla f(x)\nabla f(x)^{\top}\nonumber\\
        &\leq LQ(x)^{\top}\mathrm{Diag}(-\nabla f(x)\oslash x)Q(x)\nonumber\\
        &=L\mathrm{Diag}(-\nabla f(x)\oslash x)-L\nabla f(x)\nabla f(x)^{\top},
    \end{align}
    where the first and last equalities follow from Lemma~\ref{lemma:SelfNormalize}, and the second equality follows from Lemma~\ref{lemma:LogHomoHessianEq}.
    This concludes the proof.
\end{proof}

\begin{proofof}{Lemma~\ref{lemma:FisherCurvature}}
    The log-homogeneity of $f_{\text{Fis}}$ follows directly from the homogeneity of the power function $x\mapsto x^{\rho/(\rho-1)}$.

    The continuity of $f_{\text{Fis}}$ follows from the continuity of the logarithm and the power function $x\mapsto x^{\rho/(\rho-1)}$.

    The negativity of the gradient follows from the computation
    \begin{align}
        -\nabla f_{\text{Fis}}(x)
        =\sum_{i=1}^n w[i]\frac{a_i\odot x^{1/(\rho-1)}}{\inner{a_i}{x^{\rho/(\rho-1)}}}
        \in\mathbb{R}_{++}^m,\quad\text{for all } x\in\mathbb{R}_{++}^m.
    \end{align}

    It remains to verify the inequalities~\eqref{ineq:IntroHessianBounds}.
    We compute
    \begin{align}
        \nabla^2 f_{\text{Fis}}(x)
        =\frac{\rho}{\rho-1}H_1+\frac{1}{1-\rho}H_2,\quad\text{for all } x\in\mathbb{R}_{++}^m,
    \end{align}
    where 
    \begin{align}
        H_1
        =\sum_{i=1}^n w[i]\frac{(a_i\odot x^{1/(\rho-1)}) (a_i\odot x^{1/(\rho-1)})^{\top}}{\inner{a_i}{x^{\rho/(\rho-1)}}^2},
    \end{align}
    and 
    \begin{align}
        H_2
        =\mathrm{Diag}\left(\sum_{i=1}^n w[i]\frac{a_i\odot x^{(2-\rho)/(\rho-1)}}{\inner{a_i}{x^{\rho/(\rho-1)}}}\right)
        =\mathrm{Diag}\left(-\nabla f_{\text{Fis}}(x)\oslash x\right).
    \end{align}

    Furthermore, we have
    \begin{align}
        v^{\top} H_1 v
        &=\sum_{i=1}^n w[i]\frac{\inner{a_i\odot x^{1/(\rho-1)}}{v}^2}{\inner{a_i}{x^{\rho/(\rho-1)}}^2}\nonumber\\
        &=\sum_{i=1}^n w[i]\frac{\inner{a_i\odot x^{1/(\rho-1)}}{v}^2\inner{a_i\odot x^{(2-\rho)/(\rho-1)}}{v^2}}{\inner{a_i}{x^{\rho/(\rho-1)}}\cdot\inner{a_i}{x^{\rho/(\rho-1)}}\cdot\inner{a_i\odot x^{(2-\rho)/(\rho-1)}}{v^2}}\nonumber\\
        &\leq\sum_{i=1}^n w[i]\frac{\inner{a_i\odot x^{(2-\rho)/(\rho-1)}}{v^2}}{\inner{a_i}{x^{\rho/(\rho-1)}}}\nonumber\\
        &=v^{\top} H_2 v,\quad\text{for all } v\in\mathbb{R}^m,
    \end{align}
    where the inequality follows from the Cauchy-Schwarz inequality.
    Here, without loss of generality, we assume that $\inner{a_i\odot x^{(2-\rho)/(\rho-1)}}{v^2}\neq 0$ for all $i$; otherwise, we may remove the corresponding terms from the summation.
    This implies that $0\leq H_1\leq H_2$.

    Therefore, we have
    \begin{align}
        \min\Set{1,1/(1-\rho)}H_2
        \leq\nabla^2 f_{\text{Fis}}(x)
        \leq \max\Set{1,1/(1-\rho)}H_2.
    \end{align}
    The proof then follows by applying Lemma~\ref{lemma:RemoveRankOne}.
\end{proofof}

\subsection{Application to Kelly's Portfolio}
\label{appendix:KellyPortfolio}
In this section, to ensure that the objective function in Kelly's portfolio problem~\eqref{def:Kelly} satisfies Assumption~\ref{assumption:BasicAssumptions}, we make the following assumptions on the distribution $P$ of the price relatives $a$:
\begin{assumption}
    \label{assumption:KellyDistribution}
    \leavevmode
    \begin{itemize}
        \item We have $a\in\mathbb{R}_{+}^m\setminus\Set{0}$ almost surely.
        \item For every $x\in\mathbb{R}_{++}^m$, we have $\nabla f_{\text{Kel}}(x)=-\mathbb{E}_{a\sim P}\left[\frac{a}{\inner{a}{x}}\right]$ and $\nabla^2 f_{\text{Kel}}(x)=\mathbb{E}_{a\sim P}\left[\frac{a a^{\top} }{\inner{a}{x}^2}\right]$.
        \item For every $1\leq j\leq m$, we have $\mathbb{P}_{a\sim P}(a[j]>0)>0$.
        \item We have $\mathbb{E}_{a\sim P}[\abs{\log\left(\inner{a}{\mathbf{1}_m}\right)}]<\infty$.
    \end{itemize}
\end{assumption}
\begin{remark}
    The third condition in Assumption~\ref{assumption:KellyDistribution} 
    can be imposed without loss of generality by removing stocks
    whose price relatives are zero almost surely.
    Under the first condition, the second and fourth conditions
    hold automatically whenever $P$ has finite support, as in the finite-sum formulations of the best constant rebalanced portfolio \citep{Cover1991}
    and of positron
    emission tomography.
\end{remark}

The following lemma shows that our framework applies to the optimization problem~\eqref{def:Kelly}.
\begin{lemma}
    \label{lemma:KellyCurvature}
    The function $f_{\text{Kel}}$ in the optimization problem~\eqref{def:Kelly} satisfies Assumption~\ref{assumption:BasicAssumptions} and the second-order characterization~\eqref{ineq:IntroHessianBounds} with $(\theta,\mu,L)=(-1,0,1)$.
\end{lemma}
The proof of Lemma~\ref{lemma:KellyCurvature} is provided in Appendix~\ref{appendix:ProofKellyCurvature}.

Combining Lemma~\ref{lemma:KellyCurvature} with Theorem~\ref{theorem:ConvergenceRateMGM}, we obtain the following corollary, which recovers the convergence rate for the MGM~\eqref{alg:IntroMGM} with $\eta=1$ established by \citet{Iusem1992}, \citet{Frankowski2022}, and \citet{Lin2021}.
\begin{corollary}
    \label{cor:KellyConv}
    Let $x_t$ be the sequence generated by the MGM~\eqref{alg:IntroMGM} for solving the optimization problem~\eqref{def:Kelly} with $\eta=1$, $\theta=-1$, and $x_1=\mathbf{1}_m/m$.
    Then, for all $T$, we have
    \begin{align}
        f_{\text{Kel}}(x_{T+1})-f_{\text{Kel}}(x_{\text{Kel},\star})
        \leq \frac{\log(m)}{T}.
    \end{align}
\end{corollary}

\subsubsection{Proof of Lemma~\ref{lemma:KellyCurvature}}
\label{appendix:ProofKellyCurvature}
The proof of Lemma~\ref{lemma:KellyCurvature} is based on Lemma~\ref{lemma:KellyContinuity}.

\begin{lemma}[{[\citealp{Vajda1993}, Lemma~5]}]
    \label{lemma:KellyContinuity}
    Let $f_{\text{Kel}}$ be the function defined in the optimization problem~\eqref{def:Kelly} and suppose that Assumption~\ref{assumption:KellyDistribution} holds.
    Then, $f_{\text{Kel}}$ is continuous on $\Delta_m$.
\end{lemma}

\begin{proofof}{Lemma~\ref{lemma:KellyCurvature}}

   The log-homogeneity of $f_{\text{Kel}}$ follows directly from the homogeneity of the mapping $x\mapsto \inner{a}{x}$.

    The negativity of the gradient $\nabla f_{\text{Kel}}(x)$ on $\mathbb{R}_{++}^m$ directly follows from Assumption~\ref{assumption:KellyDistribution}.

    The continuity of $f_{\text{Kel}}$ on $\Delta_m$ follows from Lemma~\ref{lemma:KellyContinuity}.
    By log-homogeneity, we have
    $f_{\text{Kel}}(x)=f_{\text{Kel}}(x/\norm{x}_1)-\log(\norm{x}_1)$,
    so this continuity extends to $\mathbb{R}_{+}^m\setminus\Set{0}$.

    It remains to verify the second-order characterization~\eqref{ineq:IntroHessianBounds}.
    For any $x\in\mathbb{R}_{++}^m$ and $v\in\mathbb{R}^m$, we have
    \begin{align}
        \inner{\nabla f_{\text{Kel}}(x)}{v}^2
        &=\left(\mathbb{E}_{a\sim P}\left[\frac{\inner{a}{v}}{\inner{a}{x}}\right]\right)^2
        \leq\mathbb{E}_{a\sim P}\left[\frac{\inner{a}{v}^2}{\inner{a}{x}^2}\right]
        =v^{\top}\nabla^2 f_{\text{Kel}}(x)v,
    \end{align}
    where the inequality follows from Jensen's inequality.
    This implies that, for all $x\in\mathbb{R}_{++}^m$, we have $\nabla^2 f_{\text{Kel}}(x) \geq \nabla f_{\text{Kel}}(x)\nabla f_{\text{Kel}}(x)^{\top}$.

    On the other hand, by the Cauchy-Schwarz inequality,
    $\inner{a}{v}^2\leq\inner{a}{x}\inner{a\odot x^{-1}}{v^2}$, so
    \begin{align}
        v^{\top}\nabla^2 f_{\text{Kel}}(x)v
        &\leq\mathbb{E}_{a\sim P}\left[\frac{\inner{a\odot x^{-1}}{v^2}}{\inner{a}{x}}\right]
        =v^{\top}\mathrm{Diag}(-\nabla f_{\text{Kel}}(x)\oslash x)v.
    \end{align}
    This implies that $\nabla^2 f_{\text{Kel}}(x) \leq \mathrm{Diag}(-\nabla f_{\text{Kel}}(x)\oslash x)$ for all $x\in\mathbb{R}_{++}^m$.
    This concludes the proof.
\end{proofof}

\subsection{Application to D-Optimal Design}
\label{appendix:DOptimalDesign}
The D-optimal design problem is formulated as follows \citep{Kiefer1960,Atwood1969}:
\begin{align}
    \label{def:Dopt}
    x_{\text{Des},\star}
    \in \argmin_{x\in\Delta_m}f_{\text{Des}}(x),
    \quad f_{\text{Des}}(x) 
    \coloneqq -\frac{1}{n}\log(\det(A^{\top}\mathrm{Diag}(x)A)),
\end{align}
where $A\in\mathbb{R}^{ m\times n}$.
It aims at maximizing the determinant of the Fisher information matrix.
Without loss of generality, by an argument similar to that in Section~\ref{subsec:LewisWeights}, we assume that $A$ has full column rank and has no zero rows.

The following lemma shows that our framework applies to the optimization problem~\eqref{def:Dopt}.
\begin{lemma}
    \label{lemma:DoptCurvature}
    The function $f_{\text{Des}}(x)$ in the optimization problem~\eqref{def:Dopt} satisfies Assumption~\ref{assumption:BasicAssumptions} and the second-order characterization~\eqref{ineq:IntroHessianBounds} with $(\theta,\mu,L)=(-1,0,1)$.
\end{lemma}
The proof of Lemma~\ref{lemma:DoptCurvature} is provided in Appendix~\ref{appendix:ProofDoptCurvature}.

\begin{lemma}[{[\citealp{Cohen2019}, Equation~(5)]}]
    \label{lemma:DoptSmallEntry}
    We have
    \begin{align}
        x_{\text{Des},\star}[j]
        \leq \frac{1}{n},\quad \text{for all } 1\leq j\leq m.
    \end{align}
\end{lemma}

Combining Lemma~\ref{lemma:DoptCurvature}, Lemma~\ref{lemma:DoptSmallEntry}, and Theorem~\ref{theorem:ConvergenceRateMGM}, we obtain the following corollary,
which recovers the convergence rate for the MGM~\eqref{alg:IntroMGM} with $\eta=1$ established by \citet{Cohen2019} for solving the optimization problem~\eqref{def:Dopt}.
\begin{corollary}
    \label{cor:DoptConv}
    Let $x_t$ be the sequence generated by the MGM~\eqref{alg:IntroMGM} for solving the optimization problem~\eqref{def:Dopt} with  $\eta=1$, $\theta=-1$, and $x_1=\mathbf{1}_m/m$.
    Then, we have
    \begin{align}
        f_{\text{Des}}(x_{T+1})-f_{\text{Des}}(x_{\text{Des},\star})
        \leq \frac{\log\left(\frac{m}{n}\right)}{T}.
    \end{align}
\end{corollary}

\subsubsection{Proof of Lemma~\ref{lemma:DoptCurvature}}
\label{appendix:ProofDoptCurvature}
The proof of Lemma~\ref{lemma:DoptCurvature} is based on Lemma~\ref{lemma:KnownHessianBoundsDes}.

\begin{lemma}[{[\citealp{Lu2018b}, Equation~(17)]}]
    \label{lemma:KnownHessianBoundsDes}
    The objective function in the optimization problem~\eqref{def:Dopt} satisfies
    \begin{align}
        0\leq
        \nabla^2 f_{\text{Des}}(x)
        \leq \mathrm{Diag}(-\nabla f_{\text{Des}}(x)\oslash x),\quad \text{for all } x\in\mathbb{R}_{++}^m.
    \end{align}
\end{lemma}
\begin{proofof}{Lemma~\ref{lemma:DoptCurvature}}
    The log-homogeneity of $f_{\text{Des}}$ follows directly from the homogeneity of the determinant.

    The continuity of $f_{\text{Des}}(x)$ follows from the continuity of the logarithm and the determinant.

    To prove that $-\nabla f_{\text{Des}}(x)>0$ for all $x\in\mathbb{R}_{++}^m$, we compute
    \begin{align}
        -\nabla f_{\text{Des}}(x)[j]
        = \frac{1}{n}A[j]^{\top}(A^{\top}\mathrm{Diag}(x)A)^{-1}A[j]
        >0,\quad \text{for all } 1\leq j\leq m.
    \end{align}
    This is positive due to the facts that $x[j] > 0$, $(A^{\top}\mathrm{Diag}(x)A)^{-1}>0$, and $A[j] \neq 0$.

    The proof then concludes by combining Lemma~\ref{lemma:KnownHessianBoundsDes} with Lemma~\ref{lemma:RemoveRankOne}.
\end{proofof}





\subsection{Application to Lewis Weight }
\label{appendix:ApplicationLewisWeight}
We prove Lemma~\ref{lemma:HessianBoundsLewis}, Corollary~\ref{cor:MGMMultiplicativeEstimateConv}, and Lemma~\ref{lemma:OptError2LogDetDiv} in Appendix~\ref{appendix:ProofHessianBoundsLewis}, Appendix~\ref{appendix:ProofMGMMultiplicativeEstimateConv}, and Appendix~\ref{appendix:ProofOptError2LogDetDiv}, respectively.
The first two results are stated in
Section~\ref{subsec:LewisWeights}, where we apply our framework to the $\ell_p$-Lewis weight problem.
The proof of Corollary~\ref{cor:MGMMultiplicativeEstimateConv} is based on Lemma~\ref{lemma:OptError2LogDetDiv} and Lemma~\ref{lemma:HessianBoundsLewis}.

\subsubsection{Proof of Lemma~\ref{lemma:HessianBoundsLewis}}
\label{appendix:ProofHessianBoundsLewis}



\begin{proof}
    The log-homogeneity of $f_{\text{Lew}}$ follows directly from the homogeneity of the mapping $x\mapsto x^{1-2/p}$ and that of the determinant.

    The continuity for the case $p>2$ follows from the continuity of the mapping $x\mapsto x^{1-2/p}$, the continuity of the determinant, and the continuity of the logarithm.
    Continuity in the case $p< 2$ follows from the fact that the function values along sequences in $\mathbb{R}_{++}^m$ converging to the boundary of $\mathbb{R}_{++}^m$ converge to $\infty$.

    To prove that $-\nabla f_{\text{Lew}}(x)>0$ for all $x\in\mathbb{R}_{++}^m$, we compute
    \begin{align}
        -\nabla f_{\text{Lew}}(x)[j]
        = \frac{1}{n}x[j]^{-2/p}\inner{M(x)^{-1}A[j]}{A[j]}.
    \end{align}
    This is positive due to the facts that $x[j] > 0$, $M(x)^{-1}>0$, and $A[j] \neq 0$.

    The proof then concludes by combining Lemma~\ref{lemma:KnownHessianBoundsLewis} with Lemma~\ref{lemma:RemoveRankOne}.
\end{proof}

\subsubsection{Proof of Corollary~\ref{cor:MGMMultiplicativeEstimateConv}}
\label{appendix:ProofMGMMultiplicativeEstimateConv}
The proof of Corollary~\ref{cor:MGMMultiplicativeEstimateConv} is based on Lemma~\ref{lemma:FixPtProp}, Lemma~\ref{lemma:OptError2LogDetDiv}, Lemma~\ref{lemma:LogDetDiv2Lowner}, and Lemma~\ref{lemma:NesterovLogIneq}.

\begin{lemma}[{[\citealp{Lee2020}, Definition~20, Lemma~22]}]
    \label{lemma:FixPtProp}
    For each coordinate $j$, we have
    \begin{align}
        -\nabla {f}_{\text{Lew}}(x_{\text{Lew},\star})[j]
        =\frac{1}{n}\mathrm{Tr}\left(M(x_{\text{Lew},\star})^{-1}A_j\right)x_{\text{Lew},\star}[j]^{-2/p}
        =1,\nonumber
    \end{align}
    where ${f}_{\text{Lew}}$ and $x_{\text{Lew},\star}$ are defined in the optimization problem~\eqref{def:Lewis} and $A_j\coloneqq A[j]A[j]^{\top}$.
\end{lemma}

\begin{lemma}[{[\citealp{Gribling2026}, Lemma~20]}]
    \label{lemma:LogDetDiv2Lowner}
    For any $X\in\mathcal{H}_{n,++}$, if there exists $\nu\in[0,1]$ such that
    $D_{H}(X,I_n)\leq\nu\leq 0.1$,
    then
    \begin{align}
        (1-2\sqrt{\nu})I_n
        \leq  X
        \leq(1+2\sqrt{\nu})I_n,\nonumber
    \end{align}
    where $H(X)=-\log(\det(X))$.
\end{lemma}

\begin{lemma}
    \label{lemma:OptError2LogDetDiv}
    For every $x\in \mathrm{ri}(\Delta_m)$, let $Y=M(x_{\text{Lew},\star})^{-1/2}M(x)M(x_{\text{Lew},\star})^{-1/2}$.
    Then, we have
    \begin{align}
        \label{eq:matrix-divergence}
        n\left(f_{\text{Lew}}(x)-f_{\text{Lew}}(x_{\text{Lew},\star})\right)
        \geq
        D_{H}(Y^{p/(p-2)},I_n)
        ,
        \quad \text{for all } p\in(2,\infty),
    \end{align}
    where $H(X)=-\log(\det(X))$.
\end{lemma}
We defer the proof of Lemma~\ref{lemma:OptError2LogDetDiv} to Appendix~\ref{appendix:ProofOptError2LogDetDiv}.

\begin{lemma}[{[\citealp{Nesterov2018a}, Lemma~5.1.5]}]
    \label{lemma:NesterovLogIneq}
    For any $t\geq 0$, we have
    \begin{align}
        \frac{t^2}{2(1+t)}
        \leq
        t-\log(1+t)
        \leq \frac{t^2}{2+t}.\nonumber
    \end{align}
    On the other hand, for any $t\in[0,1)$, we have
    \begin{align}
        \frac{t^2}{2-t}
        \leq
        -t-\log(1-t)
        \leq \frac{t^2}{2(1-t)}.\nonumber
    \end{align}
\end{lemma}

\begin{proofof}{Corollary~\ref{cor:MGMMultiplicativeEstimateConv}}
    By Corollary~\ref{cor:LewisMGMOptErrorKLConv}, we have
    \begin{align}  
        n\left(f_{\text{Lew}}(x_{T+1})-f_{\text{Lew}}(x_{\text{Lew},\star})\right)
        &\leq \left(1-\frac{2}{p}\right)^Tn\log\left(\frac{m}{n}\right)\nonumber\\
        &\leq \exp(-\frac{2}{p}T)n\log\left(\frac{m}{n}\right)\nonumber\\
        &\leq  \exp\left(-\frac{2}{p}\frac{p}{2}\log\left(16\frac{np^2}{\varepsilon^2}\log\left(\frac{m}{n}\right)\right)\right)n\log\left(\frac{m}{n}\right)\nonumber\\
        &= \frac{\varepsilon^2}{16p^2}.\nonumber
    \end{align}
    Therefore, by Lemma~\ref{lemma:OptError2LogDetDiv}, we have
    \begin{align}
        \Tr\left(Y^{p/(p-2)}\right)-\log\left(\det\left(Y^{p/(p-2)}\right)\right)-n
        &\leq n\left( f_{\text{Lew}}(x_{T+1})-f_{\text{Lew}}(x_{\text{Lew},\star})\right)\nonumber\\
        &\leq \frac{\varepsilon^2}{16p^2}\nonumber\\
        &\leq 0.1,\nonumber
    \end{align}
    where $Y$ is defined in Lemma~\ref{lemma:OptError2LogDetDiv}.
    Furthermore, by Lemma~\ref{lemma:LogDetDiv2Lowner}, we obtain
    \begin{align}
        \left(1-\frac{\varepsilon}{2p}\right)I_n
        \leq Y^{p/(p-2)}
        \leq \left(1+\frac{\varepsilon}{2p}\right)I_n.\nonumber
    \end{align}
    Equivalently, we have
    \begin{align}
        \left(1+\frac{\varepsilon}{2p}\right)^{-(1-2/p)} M(x_{\text{Lew},\star})^{-1}
        \leq M(x_{T+1})^{-1}
        \leq \left(1-\frac{\varepsilon}{2p}\right)^{-(1-2/p)} M(x_{\text{Lew},\star})^{-1}.\nonumber
    \end{align}
    By Lemma~\ref{lemma:FixPtProp}, we have
    \begin{align}
        \frac{1}{n}\left(1+\frac{\varepsilon}{2p}\right)^{-(1-2/p)} \Tr\left(M(x_{\text{Lew},\star})^{-1}A_j\right)
        &=\left(1+\frac{\varepsilon}{2p}\right)^{-(1-2/p)} x_{\text{Lew},\star}[j]^{2/p}\nonumber\\
        &\leq \frac{1}{n}\Tr\left(M(x_{T+1})^{-1}A_j\right)\nonumber\\
        &=\left(x_{T+1}[j]\left(-\nabla f_{\text{Lew}}(x_{T+1})[j]\right)^{p/2}\right)^{2/p}\nonumber\\
        &\leq \left(1-\frac{\varepsilon}{2p}\right)^{-(1-2/p)} x_{\text{Lew},\star}[j]^{2/p},\nonumber
    \end{align}
    where $A_j$ is defined as in Lemma~\ref{lemma:FixPtProp}.
    By raising both sides to the power $p/2$, we obtain
    \begin{align}
        \left(1+\frac{\varepsilon}{2p}\right)^{-(p/2-1)} x_{\text{Lew},\star}[j]
        &\leq x_{T+1}[j]\left(-\nabla f_{\text{Lew}}(x_{T+1})[j]\right)^{p/2}\nonumber\\
        &\leq \left(1-\frac{\varepsilon}{2p}\right)^{-(p/2-1)} x_{\text{Lew},\star}[j].\nonumber
    \end{align}
    
    The proof continues as follows:
    \begin{align}
        \left(1+\frac{\varepsilon}{2p}\right)^{-(p/2-1)} 
        &=\exp\left(-\left(\frac{p}{2}-1\right)\log\left(1+\frac{\varepsilon}{2p}\right)\right)\nonumber\\
        &\geq\exp\left(-\left(\frac{p}{2}-1\right)\frac{\varepsilon}{2p}\right)\nonumber\\ 
        &\geq 1-\frac{\varepsilon}{4}\nonumber\\
        &\geq 1-\varepsilon.\nonumber
    \end{align}
    On the other hand, we write
    \begin{align}
        \left(1-\frac{\varepsilon}{2p}\right)^{-(p/2-1)} 
        &=\exp\left(-\left(\frac{p}{2}-1\right)\log\left(1-\frac{\varepsilon}{2p}\right)\right)\nonumber\\
        &\leq\exp\left(\left(\frac{p}{2}-1\right)\left(\frac{2\left(\frac{\varepsilon}{2p}\right)-\left(\frac{\varepsilon}{2p}\right)^2}{2\left(1-\left(\frac{\varepsilon}{2p}\right)\right)}\right)\right)\nonumber\\ 
        &\leq\exp\left(\left(\frac{p}{2}-1\right)\left(\frac{2\left(\frac{\varepsilon}{2p}\right)}{2\left(1-\left(\frac{\varepsilon}{2p}\right)\right)}\right)\right)\nonumber\\ 
        &\leq\exp\left(\left(\frac{p}{2}-1\right)\frac{\varepsilon}{p}\right)\nonumber\\ 
        &\leq \exp\left(\frac{\varepsilon}{2}\right)\nonumber\\
        &\leq \exp\left(\frac{2\varepsilon}{2+\varepsilon}\right)\nonumber\\
        &=\exp\left(\varepsilon-\frac{\varepsilon^2}{2+\varepsilon}\right)\nonumber\\
        &\leq \exp\left(\log(1+\varepsilon)\right)\nonumber\\
        &= 1+\varepsilon,\nonumber
    \end{align}
    where the first and the last inequalities follow from Lemma~\ref{lemma:NesterovLogIneq}, and the third and the fifth inequalities follow from the fact that $\varepsilon\in(0,1]$.
    This concludes the proof.
\end{proofof}

\subsubsection{Proof of Lemma~\ref{lemma:OptError2LogDetDiv}}
\label{appendix:ProofOptError2LogDetDiv}
We will use Jensen's trace inequality (Lemma~\ref{lemma:JensenTraceIneq}) to prove Lemma~\ref{lemma:OptError2LogDetDiv}.

\begin{lemma}[{[\citealp{Hansen2003}, Theorem~2.4]}]
    \label{lemma:JensenTraceIneq}
    Suppose that a function $f$ is continuous and convex on $\mathbb{R}_{++}$.
    Then, for any matrices $X_1,\dots,X_L\in\mathcal{H}_{n,++}$ and $C_1,\dots,C_L\in\mathcal{H}_{n}$ such that $\sum_{i=1}^L C_i^2=I_n$, we have
    \begin{align}
        \Tr\left(f\left(\sum_{i=1}^L C_iX_iC_i\right)\right)
        \leq \Tr\left(\sum_{i=1}^L C_if(X_i)C_i\right).\nonumber
    \end{align}
\end{lemma}

\begin{proofof}{Lemma~\ref{lemma:OptError2LogDetDiv}}
    We define $E_j$ by
    \begin{align}
        \label{def:Ej}
        E_j 
        =
        M(x_{\text{Lew},\star})^{-0.5}\left(x_{\text{Lew},\star}[j]^{-2/p}A_j\right)M(x_{\text{Lew},\star})^{-0.5},\quad\text{for all } 1\leq j\leq m,
    \end{align}
    and let $A_j$ be defined as in Lemma~\ref{lemma:FixPtProp}.
    We have
    \begin{align}
        \Tr\left(Y^{p/(p-2)}\right)
        &=\Tr\left(\left(M(x_{\text{Lew},\star})^{-0.5}M(x)M(x_{\text{Lew},\star})^{-0.5}\right)^{p/(p-2)}\right)\nonumber\\
        &=\Tr\left(\left(\sum_{j=1}^m\left(\frac{x[j]}{x_{\text{Lew},\star}[j]}\right)^{1-2/p}x_{\text{Lew},\star}[j]E_j\right)^{p/(p-2)}\right)\nonumber\\
        &=\Tr\left(\left(\sum_{j=1}^m\sqrt{x_{\text{Lew},\star}[j]E_j}\left(\frac{x[j]}{x_{\text{Lew},\star}[j]}I_n\right)^{1-2/p}\sqrt{x_{\text{Lew},\star}[j]E_j}\right)^{p/(p-2)}\right).\nonumber
    \end{align}
    In addition, we have $((x[j]/x_{\text{Lew},\star}[j])I_n)^{1-2/p}\in\mathcal{H}_{n,++}$ and $\sum_{j=1}^m \sqrt{x_{\text{Lew},\star}[j]E_j}^2=I_n$.
    Therefore, by Jensen's trace inequality (Lemma~\ref{lemma:JensenTraceIneq})  and the convexity of the mapping $x\mapsto x^{p/(p-2)}$ on $(0,\infty)$, we have
    \begin{align}
        \Tr\left(Y^{p/(p-2)}\right)
        &=\Tr\left(\left(\sum_{j=1}^m\sqrt{x_{\text{Lew},\star}[j]E_j}\left(\frac{x[j]}{x_{\text{Lew},\star}[j]}I_n\right)^{1-2/p}\sqrt{x_{\text{Lew},\star}[j]E_j}\right)^{p/(p-2)}\right)\nonumber\\
        &\leq\Tr\left(\sum_{j=1}^m\sqrt{x_{\text{Lew},\star}[j]E_j}\left(\left(\frac{x[j]}{x_{\text{Lew},\star}[j]}I_n\right)^{1-2/p}\right)^{p/(p-2)}\sqrt{x_{\text{Lew},\star}[j]E_j}\right)\nonumber\\
        &=\sum_{j=1}^m x[j]\mathrm{Tr}(E_j)\nonumber\\
        &=n\sum_{j=1}^m x[j]\nonumber\\
        &=n,\nonumber
    \end{align}
    where the third equality follows from Lemma~\ref{lemma:FixPtProp}.

    Hence, we have
    \begin{align}
        n\left(f_{\text{Lew}}(x)-f_{\text{Lew}}(x_{\text{Lew},\star})\right)
        &=\frac{1}{\frac{2}{p}-1}\left(\log\left(\det(M(x))\right)-\log\left(\det(M(x_{\text{Lew},\star}))\right)\right)\nonumber\\
        &=n-\log\left(\det\left(Y^{p/(p-2)}\right)\right)-n\nonumber\\
        &\geq \Tr\left(Y^{p/(p-2)}\right)-\log\left(\det\left(Y^{p/(p-2)}\right)\right)-n.\nonumber
    \end{align}
    This concludes the proof.
\end{proofof}

\subsection{Application to Petz-R\'{e}nyi Capacity}
\label{appendix:ApplicationPetzRenyiCapacity}
We prove Lemma~\ref{lemma:HessianBoundsPetz} and
Corollary~\ref{cor:PetzMGMOptErrorKLConv} in
Appendix~\ref{appendix:ProofHessianBoundsPetz} and
Appendix~\ref{appendix:ProofPetzMGMOptErrorKLConv}, respectively.
Both results are stated in Section~\ref{subsec:PetzRenyiCapacity}.

\subsubsection{Proof of Lemma~\ref{lemma:HessianBoundsPetz}}
\label{appendix:ProofHessianBoundsPetz}
The proof of Lemma~\ref{lemma:HessianBoundsPetz} is based on Lemma~\ref{lemma:TiltedTraceJointConcave} and Lemma~\ref{lemma:MatrixDifferentiable}.

\begin{lemma}[{[\citealp{Carlen2008}, Theorem~1.1]}]
    \label{lemma:TiltedTraceJointConcave}
    The mapping 
    \begin{align}
        (X_1,\dots,X_m)\mapsto \Tr\left(\left(\sum_{j=1}^m X_j^{\alpha}\right)^{1/\alpha}\right)
    \end{align}
    is concave on $\mathcal{H}_{n,+}^m$ for $0<\alpha<1$.
\end{lemma}

\begin{lemma}[{[\citealp{Hiai2014}, Theorem~3.33]}]
    \label{lemma:MatrixDifferentiable}
    For any function $f\colon\mathbb{R}\to\mathbb{R}$ that is $m$ times continuously differentiable on an open interval $(a,b)\subseteq\mathbb{R}$, its matrix extension $f\colon \mathcal{H}_n\to\mathcal{H}_n$ is also $m$ times Fr\'{e}chet differentiable on the set of Hermitian matrices with eigenvalues in $(a,b)$.
\end{lemma}


\begin{proofof}{Lemma~\ref{lemma:HessianBoundsPetz}}
    Without loss of generality, we assume that
    $\sum_{j=1}^m\mathcal{W}(j)$ is positive definite; otherwise, we may project all quantum states $\mathcal{W}(j)$ onto the support of this sum, which does not change $f_{\text{Petz}}$.

    The log-homogeneity of $f_{\text{Petz}}$ follows directly from the homogeneity of the Schatten norm.

    The continuity of $f_{\text{Petz}}$ follows from the continuity of the Schatten norm and the fact that $\mathrm{Tr}(W(x)^{1/\alpha})>0$ for all $x\in\mathbb{R}_{+}^m\setminus\Set{0}$.

    The fact that $f_{\text{Petz}}$ is twice differentiable on $\mathbb{R}_{++}^m$ follows from Lemma~\ref{lemma:MatrixDifferentiable}, the fact that $x\mapsto W(x)$ is linear, and the fact that the function $x\mapsto x^{1/\alpha}$ is twice continuously differentiable on $\mathbb{R}_{++}$.

    To prove that $\nabla f_{\text{Petz}}(x)>0$ for all $x\in\mathbb{R}_{++}^m$, 
    we compute
    \begin{align}
        \nabla f_{\text{Petz}}(x)[j]
        =\frac{\mathrm{Tr}(W(x)^{1/\alpha-1}\mathcal{W}(j)^{\alpha})}{\mathrm{Tr}(W(x)^{1/\alpha})},
    \end{align}
    which is positive because $W(x)\in\mathcal{H}_{n,++}$ and $\mathcal{W}(j)\in\mathcal{D}_{n}$.

    By Lemma~\ref{lemma:ZeroOrderCharacterization}, it remains to establish the concavity of $U_{1,1/\alpha-1}^{f_{\text{Petz}}}(y)=\alpha\mathrm{Tr}\left(W(y^{\alpha})^{1/\alpha}\right)$ and the convexity of $U_{1,0}^{f_{\text{Petz}}}(y)=\mathrm{Tr}\left(W(y)^{1/\alpha}\right)^{\alpha}$.

    The convexity of $U_{1,0}^{f_{\text{Petz}}}(y)$ follows directly from the convexity of the Schatten norm.

    For the concavity of $\mathrm{Tr}\left(W(y^{\alpha})^{1/\alpha}\right)$, we note that it can be expressed as a composition of two mappings $\Phi_1\colon \mathcal{H}_{n,+}^m\to \mathbb{R}_{+}\colon (A_1,\dots,A_m)\mapsto \mathrm{Tr}\left(\left(\sum_{j=1}^m A_j^{\alpha}\right)^{1/\alpha}\right)$ and $\Phi_2\colon \mathbb{R}_{+}^m\to \mathcal{H}_{n,+}^m\colon y\mapsto (y[1]\mathcal{W}(1),\dots,y[m]\mathcal{W}(m))$.
    The mapping $\Phi_1$ is concave due to Lemma~\ref{lemma:TiltedTraceJointConcave}, and the mapping $\Phi_2$ is linear; hence $U_{1,1/\alpha-1}^{f_{\text{Petz}}}(y)$ is concave.
    This concludes the proof.
\end{proofof}

\subsubsection{Proof of Corollary~\ref{cor:PetzMGMOptErrorKLConv}}
\label{appendix:ProofPetzMGMOptErrorKLConv}
The proof of Corollary~\ref{cor:PetzMGMOptErrorKLConv} is based on Lemma~\ref{lemma:PetzCapacityBounds},  Lemma~\ref{lemma:HessianBoundsPetz}, and  Theorem~\ref{theorem:ConvergenceRateMGM}.
\begin{lemma}[{[\citealp{Cheng2019}, Equation~(13), Proposition~2]}]
    \label{lemma:PetzCapacityBounds}
    The function value $f_{\text{Petz}}(x)$ in the optimization problem~\eqref{def:PetzRenyiCapacity} is non-positive and bounded below by $(\alpha-1)\log(m)$ for all $x\in\Delta_m$.
\end{lemma}
\begin{proofof}{Corollary~\ref{cor:PetzMGMOptErrorKLConv}}
    Let $x_{\text{Petz},\star}$ be a minimizer of $f_{\text{Petz}}(x)$ over $x\in\Delta_m$.
    By Lemma~\ref{lemma:HessianBoundsPetz} and Theorem~\ref{theorem:ConvergenceRateMGM}, we have
    \begin{align}
        0
        &\leq f_{\text{Petz}}(x_{T+1})-f_{\text{Petz}}(x_{\text{Petz},\star})\nonumber\\
        &\leq \frac{(1-\alpha)(\RenyiKL{x_{\text{Petz},\star}}{x_1}+f_{\text{Petz}}(x_1)-f_{\text{Petz}}(x_{\text{Petz},\star}))}{\alpha T}\nonumber\\
        &\leq \frac{(1-\alpha)(\log(m)+f_{\text{Petz}}(x_1)-f_{\text{Petz}}(x_{\text{Petz},\star}))}{\alpha T}\nonumber\\
        &\leq \frac{(1-\alpha)(\log(m)+(1-\alpha)\log(m))}{\alpha T},\nonumber
    \end{align}
    where the last inequality follows from Lemma~\ref{lemma:PetzCapacityBounds}.
    This concludes the proof.
\end{proofof}

\subsection{Comparison of Error Measures}
\label{appendix:ComparisonErrorMeasures}
For computing market equilibria~\eqref{def:Fisher} and $\ell_p$-Lewis weights~\eqref{def:Lewis}, \citet{Chu2025} and \citet{Cohen2015b} established convergence guarantees for the MGM in terms of the Thompson metric, defined as
\begin{align}
    d_{\text{T}}(x,y) = \norm{\log(x)-\log(y)}_{\infty}.
\end{align}
\citet{Chu2025} established that, for the equilibrium problem, the Thompson metric upper bounds the optimization error.
This appendix shows that this argument holds more generally for $\theta$-log-homogeneous objective functions.

In particular, let $r=\norm{\log(x)-\log(x_{\star})}_{\infty}$.
We observe that $\exp(-r)x \leq x_{\star} \leq \exp(r)x$, where the inequalities are understood entry-wise.
Therefore, we have
\begin{align}
    \label{ineq:Thompson2Unnormalized}
    f(x)-f(x_{\star})
    &\leq f(x)-f(\exp(-\theta r)x)\nonumber\\
    &=\theta\log(\exp(\theta r))\nonumber\\
    &=r,
\end{align}
where the inequality follows from the assumption that $\theta \nabla  f(x)>0$ for all $x\in\mathbb{R}_{++}^m$, and the first equality follows from the log-homogeneity of $f$.

On the other hand, since \citet{Cohen2015b} did not consider normalizing the iterates as in our MGM update~\eqref{alg:IntroMGM}, we further show that normalization affects the upper bound provided by the Thompson metric only up to a constant multiplicative factor.
In particular, we have
\begin{align}
    \label{ineq:Unnormalized2Normalized}
    f\left(\frac{x}{\norm{x}_1}\right)
    &=f(x)-\theta\log(\norm{x}_1)\nonumber\\
    &\leq f(x)+\abs{\theta\log(\exp(r)\norm{x_{\star}}_1)}\nonumber\\
    &=f(x)+r.
\end{align}

Combining inequalities~\eqref{ineq:Thompson2Unnormalized} and~\eqref{ineq:Unnormalized2Normalized}, we obtain 
\begin{align}
    f\left(\frac{x}{\norm{x}_1}\right) - f(x_{\star})
    \leq 2r,
\end{align}
which implies that the Thompson metric upper bounds the optimization error.

\end{document}